\documentclass[11pt, reqno]{amsart}

\usepackage[margin=2cm]{geometry}

\usepackage{amssymb}
\usepackage[OT2, T1]{fontenc}
\usepackage[usenames,dvipsnames]{color}
\usepackage{amsmath}
\usepackage{amsthm}
\usepackage{mathabx}
\usepackage[mathscr]{eucal}
\usepackage[absolute]{textpos} 
\usepackage{colortbl}
\usepackage{array}
\usepackage{geometry}
\usepackage{hyperref}
\usepackage{epigraph}
\usepackage{amscd}
\usepackage[all,cmtip]{xy}
\usepackage{tikz-cd}
\usepackage{enumitem}
\usepackage{dynkin-diagrams}
\usepackage{cleveref}
\usepackage{mathrsfs}
\usepackage[new]{old-arrows}

\newtheoremstyle{thms}{0.2em}{0.2em}{\itshape}{}{\bfseries}{.}{ }{}
\theoremstyle{thms}

\usepackage[american]{babel}

\usepackage[citestyle=alphabetic,bibstyle=alphabetic]{biblatex}

\newtheoremstyle{thms}{0.2em}{0.2em}{\itshape}{}{\bfseries}{.}{ }{}

\theoremstyle{plain}

\theoremstyle{definition}

\newtheorem{theorem}{Theorem}[section]
\newtheorem{lemma}[theorem]{Lemma}
\newtheorem{corollary}[theorem]{Corollary}
\newtheorem{definition}[theorem]{Definition}
\newtheorem{proposition}[theorem]{Proposition}
\newtheorem{remark}[theorem]{Remark}

\newtheorem{definition-proposition}[theorem]{Definition-Proposition}

\makeatletter
\newcommand*{\relrelbarsep}{.386ex}
\newcommand*{\relrelbar}{%
  \mathrel{%
    \mathpalette\@relrelbar\relrelbarsep
  }%
}
\newcommand*{\@relrelbar}[2]{%
  \raise#2\hbox to 0pt{$\m@th#1\relbar$\hss}%
  \lower#2\hbox{$\m@th#1\relbar$}%
}
\providecommand*{\rightrightarrowsfill@}{%
  \arrowfill@\relrelbar\relrelbar\rightrightarrows
}
\providecommand*{\leftleftarrowsfill@}{%
  \arrowfill@\leftleftarrows\relrelbar\relrelbar
}
\providecommand*{\xrightrightarrows}[2][]{%
  \ext@arrow 0359\rightrightarrowsfill@{#1}{#2}%
}
\providecommand*{\xleftleftarrows}[2][]{%
  \ext@arrow 3095\leftleftarrowsfill@{#1}{#2}%
}
\makeatother

\DeclareMathOperator{\Aut}{\mathrm{Aut}}		

\DeclareMathOperator{\Diag}{Diag}
\DeclareMathOperator{\End}{End} 

\DeclareMathOperator{\GL}{GL}	
	
\DeclareMathOperator{\Hom}{Hom}

\DeclareMathOperator{\Out}{Out}	

\DeclareMathOperator{\Gal}{Gal}	

\DeclareMathOperator{\Grp}{Grp}	
  
\DeclareMathOperator{\Spec}{Spec}		

\DeclareMathOperator{\Stab}{Stab}

\title{Relative Bruhat decomposition of wonderful compactification}
\author{Fei Chen and Shang Li}
\address{Tsinghua University, Yau Mathematical Sciences Center, Beijing, China}
\email{fchen@tsinghua.edu.cn}
\address{Tsinghua University, Yau Mathematical Sciences Center, Beijing, China}
\email{shangli@tsinghua.edu.cn}
\date{\today}

\begin{document}

\setcounter{tocdepth}{1}

\maketitle

\begin{abstract}
    In the seminal paper of Borel and Tits about reductive groups, they show some fundamental results about Bruhat cells with respect to a minimal parabolic subgroup, e.g., relative Bruhat decomposition and its geometrization, relative Bruhat order and the relation of Zariski closure and topological closure.
    In this paper, we show analogous results for Bruhat cells of wonderful group compactification in the sense of De Concini and Procesi.
    Our results can be viewed as the version at infinity of those of Borel and Tits. Our main focus is general base field. When the base field is algebraically closed, most of our results are proved by Brion and Springer.
\end{abstract}

\tableofcontents

\pagestyle{plain}

\section{Introduction}

Let $G$ be a connected reductive group over a (not necessarily algebraically closed) field $k$ containing a maximal $k$-split $k$-torus $S$. Let $P\subset G$ be a minimal parabolic subgroup containing $S$, and let $W$ be the relative Weyl group. 

In the Borel--Tits theory for reductive groups over arbitrary base fields, the following version of Bruhat decomposition plays an important role.

\begin{theorem}\label{RelativeBruhatforgroup}(\cite[théorème~5.15]{boreltits}, \cite[théorème~3.13, corollaire~3.15]{boreltitscomplement})

(1). (Relative Bruhat decomposition) $G(k)$ is the disjoint union of $P(k)\cdot w\cdot P(k)$ with $w\in W$.

(2). (Geometric refinement of (1)) If $P\cdot w\cdot P$ meets $P\cdot w'\cdot P$ inside $G$, then $w=w'$.

(3). (Relative Bruhat order) If $k$ is infinite, we have
$$\overline{P\cdot w \cdot P}(k)=\overline{P(k)\cdot w\cdot P(k)}\bigcap G(k)=\bigcup_{w'\leq w}P(k)\cdot w'\cdot P(k).$$

(4). (Topological closure) Suppose that $k$ is a topological field equipped with a non-discrete topology satisfying the separation axiom $T_1$ and that $G(k)$ is endowed with a topology $\mathcal{T}$ induced from that of $k$. Then the closure of $P(k)\cdot w\cdot P(k)$ with respect to $\mathcal{T}$ coincides with the relative Zariski closure $\overline{P(k)\cdot w\cdot P(k)}\bigcap G(k)$.
\end{theorem}

The Bruhat decomposition dates back to Gauss elimination method in linear algebra and plays a crucial role in group theory and representation theory, see \cite{lusztigbruhatdecompositionapplications} for a survey. Prior to \cite{boreltits},
the proofs of \Cref{RelativeBruhatforgroup} (1) are given by Harish-Chandra and Bruhat in \cite{Hrishchandra} and \cite{Bruhat} respectively when $k=\mathbb{C}$. Chevalley gives a proof valid for arbitrary algebraically closed base field in \cite{chevalley55}.

The main goal of this article is to show analogous results for wonderful group compactification, which is constructed by De~Concini ,Procesi \cite{completesymmetricvarieties} and Strickland \cite{strickland} over algebraically closed fields. The wonderful compactification for a \emph{non-split group} is obtained by \emph{Galois descent}. 
 
To state our results, we now assume that $G$ is adjoint, i.e., the center of $G$ is trivial. Let $\overline{G}$ be the wonderful compactification of $G$. Let $P^-\subset G$ be the opposite minimal parabolic subgroup such that $P\bigcap P^-=Z_G(S)$. Let $\Delta\subset X^{\ast}(S)$ be the set of (relative) simple roots corresponding to $P$. Let $l:W\rightarrow \mathbb{Z}$ be the usual length function. As an analogue of \Cref{RelativeBruhatforgroup} (1), we prove the following result, which is due to Brion when $k=\mathbb{C}$ \cite[\S, 2.1]{behaviouratinfinityofbruhat}.
\begin{theorem}\label{introBruhatdecomposition}(\Cref{propositionofbasepoint}, \Cref{borbitproposition})
    For each $P(k)\times P^-(k)$-orbit $K$ in $\overline{G}(k)$, there are unique $I\subset \Delta$, $\sigma, \tau\in W^I$, $\rho\in W_I$ and a unique base point $b\in K$ satisfying
    \begin{itemize}
        \item $(P\times_k P^-)\cdot b$ is open in $(G\times_k G)\cdot b$ and
        \item $b=\lim\limits_{x \to 0} \lambda(x)$ for a cocharacter $\lambda\in X_{\ast}(S)$
    \end{itemize}
    so that we have
    $K=(P(k)\times P^-(k))\cdot (\sigma \rho,\tau)\cdot b .$
\end{theorem}

As an analogue of \Cref{RelativeBruhatforgroup} (2), we have the following result which should be viewed as a geometric refinement of \Cref{introBruhatdecomposition}.

\begin{theorem}\label{introgeometrization}(\Cref{theoremofgeometrization})   
   Given two $P(k)\times P^-(k)$-orbits \[K_i\coloneq (P(k)\times_k P^{-}(k))\cdot (\sigma_i \rho_i,\tau_i)\cdot b_{i}\subset \overline{G}(k),\] where $b_i\in \overline{G}(k)$, $I_i\subset \Delta$, $\rho_i\in W_{I_i}$ and $\sigma_i, \tau_i\in W^{I_i}$ as in \Cref{introBruhatdecomposition} for $i=1,2$. Consider their ``geometrizations'' $O_i\coloneq (P\times_k P^{-})\cdot (\sigma_i \rho_i,\tau_i)\cdot b_{i}\subset \overline{G},\; i={1,2}.$
   If $O_1\bigcap O_2\neq \emptyset$, then $b_1=b_2$, $I_1=I_2$, $(\sigma_1,\tau_1)=(\sigma_2,\tau_2)$ and $\rho_1=\rho_2$.
\end{theorem}

As a generalization of the relative Bruhat order of $G$ (\Cref{RelativeBruhatforgroup} (3)), we show the following result which is due to Springer (based on Brion's work) when $k$ is algebraically closed.

\begin{theorem}(\Cref{relbruhatmainthm})
    Given 
    $$O_1=(P\times_kP^-)\cdot(\sigma_1\rho_1,\tau_1)\cdot b_{1}\;\;\text{and}\;\;
    O_2=(P\times_kP^-)\cdot(\sigma_2\rho_2,\tau_2)\cdot b_{2},$$
where $b_i\in \overline{G}(k)$, $\sigma_i,\tau_i\in W^{I_i}$ and $\rho_i\in W_{I_i}$ for $i=1,2$ as in \Cref{introgeometrization}.
Then $O_1\subset \overline{O_2}$ if and only if $I_1\subset I_2$ and there exist $v\in W_{I_2}\bigcap W^{I_1}$, $u\in W_{I_1}$ such that 
$\sigma_1\rho_1u\geq \sigma_2\rho_2v \text{, } \tau_1\geq \tau_2vu^{-1}$
and $l(\rho)=l(\rho v)+l(v)$.
\end{theorem}

Finally, when the base field $k$ is a topological ring which is endowed with a non-discrete topology $\mathcal{E} $ satisfying
the separation axiom $T_1$, we show the following result as a generalization of \Cref{RelativeBruhatforgroup} (4).

\begin{theorem}(\Cref{theoremtopologicalclosure})
    Let $O\subset G(k)$ be a $P(k)\times P^-(k)$-orbit. Then the closure of $O(k)$ with respect to the topology of $\overline{G}(k)$ induced from that of $k$ equals to the relative Zariski closure $\overline{O}\bigcap \overline{G}(k)$, where the bar indicates taking Zariski closure in $\overline{G}$.
\end{theorem}

\subsection{Conventions and notation}

A topological space $Y$ satisfies \emph{$T_1$ axiom} if, for any two distinct point $x,y \in Y$, there exist two open subsets $U, V\subset Y$ such that $x\in U, y\in V, x\notin V$ and $y\notin U$.

Let $X$ be an algebraic variety over a field $k$, and let $W,Z\subset X$ be two subvarieties. For an irreducible component $C$ of $W\bigcap Z$, we say that $W$ and $Z$ \emph{intersect properly} along $C$ in $X$ if 
$$\dim(C)+ \dim(X)= \dim(W)+\dim(Z).$$
We say that $W$ and $Z$ intersect properly in $X$ if $W$ and $Z$ intersect properly along all irreducible components of $W\bigcap Z$. 

An algebraic variety $X$ over a field $k$ is called \emph{$k$-dense} if the set of $k$-points of $X$ is Zariski dense in $X$.

For a $k$-morphism $f:\mathbb{G}_{m,k}\coloneq \Spec(k[x, x^{-1}])\rightarrow X$, if $f$ (necessarily uniquely) extends to a $k$-morphism $\mathbb{A}_{1,k}\rightarrow X$, we denote the image of the origin point by $\lim\limits_{x \to 0} f(x)$. 

For an affine scheme $X\coloneq \Spec(R)$ and an ideal $I\subset R$, we denote by $V_X(I)$ the closed subscheme of $X$ defined by $I$.

Let $G$ be an affine algebraic group over a field $k$, we denote by $\mathscr{R}_u(G)$ its ($k$-)unipotent radical.

For an algebraic group $G$ together with a right action on an algebraic variety $X$ over a field $k$ and a left action on an algebraic variety $Y$, we let $X\times^G_k Y$ be the contracted product, i.e., the quotient of $X\times_k Y$ by the action by $g\cdot(x,y)=(x\cdot g^{-1},g\cdot y)$.

\subsection{Acknowledgments}

The second-named author thanks Bertrand Rémy and Ning Guo for helpful conversations. The first-named author is supported by NSFC, Grant No. 12571027. The second-named author has received funding from National Key R\&D Program of China (grant 2024YFA1014700).

\section{Wonderful compactification}\label{sectionwonderfulcompactification}

In this section, we first fix some group theoretic notations in \Cref{groupsetup}. Then, we recollect some fundamental facts of wonderful compactification for split groups in \Cref{sectionsplitcase} which will be heavily used in the rest of this paper. We also review the Galois descent for wonderful compactification in order to compactify non-split groups in \Cref{galoisdescent}.

\subsection{Group setup}\label{groupsetup}

Let $G$ be an adjoint reductive group over a (not necessarily algebraically closed) field $k$. Let $k_s\subset \overline{k}$ be the separable closure of $k$ which is contained in an algebraic closure of $k$, and let $\Gamma\coloneq \Gal(k_s/k)$. Then $G_{k_s}$ is a split adjoint reductive group, see, for instance, \cite[exposé~XXII, corollaire~2.3]{SGA3III}.

\subsubsection{Absolute root system}
We fix a maximal $k$-split torus $S\subset G$.
Let $T\subset G$ be a maximal torus containing $S$ (which exists by Grothendieck's \cite[exposé~XIV, théorème~1.1]{SGA3II}). We denote the absolute root system of $G$ by 
$$\widetilde{\Phi}\subset X^{\ast}(T)\coloneq \Hom_{k_s-\Grp}(T_{k_s},\mathbb{G}_{m,k_s}).$$
The Galois group $\Gamma$ naturally acts on $X^{\ast}(T)$ so that $(\gamma\cdot \alpha)(t)= \gamma(\alpha(\gamma^{-1}t))$ for $\alpha\in X^{\ast}(T)$, $\gamma\in \Gamma$ and $t\in T(k_s)$. Let $\widetilde{W}\coloneq N_{G_{k_s}}(T_{k_s})/T_{k_s}$ be the absolute Weyl group of $G$.

Let $P\subset G$ be a minimal parabolic subgroup containing $S$. We choose a Borel subgroup $B\subset P_{k_s}$ such that $T_{k_s}\subset B$. Such choice of $B$ gives rise to a set of simple roots and the positive system 
$$\widetilde{\Delta}\subset \widetilde{\Phi}^+ \subset \widetilde{\Phi}.$$
Then, by \cite[corollaire~4.16]{boreltits}, we have the Levi decomposition
$$P=Z_G(S) \ltimes U^+,$$
where $U^+$ is the unipotent radical of $P$. We let $B^-\supset T_{k_s}$ be the opposite Borel subgroup of $B$ and $P^-\subset G$ be the opposite parabolic subgroup of $P$ such that $B^-\subset P^-_{k_s}$.

\subsubsection{Relative root system}
We denote the relative root system of $G$, the positive system corresponding to $P$ and the set of simple roots by
$$\Phi=\Phi(G,S)\supset \Phi^+\supset \Delta.$$

We define 
$$\widetilde{\Delta}_0\coloneq \{r\in \widetilde{\Delta}\vert\; r\vert_{S}=0\}\;\;\text{and}\;\;\widetilde{\Phi}_0\coloneq\{r\in\widetilde{\Phi}\vert\; r\vert_{S}=0\}.$$
Then $\widetilde{\Phi}_0$ is the root system of $Z_G(S)_{k_s}$ with respect to the maximal torus $T_{k_s}$ (\cite[Proposition~15.5.3 (ii)]{linearalgrpSpringer}).
Let $\widetilde{W}_0\subset \widetilde{W}$ be the subgroup generated by the simple reflections defined by $\widetilde{\Delta}_0$.

For a subset $J\subset \Delta$, let $\Phi^+_J\coloneq (\sum_{\alpha\in J}\mathbb{Z}_{\geq 0}\alpha)\bigcap \Phi^+
$ and $\Phi^-_J\coloneq (\sum_{\alpha\in J}\mathbb{Z}_{\leq 0}\alpha)\bigcap \Phi^-
$, where $\Phi^-\coloneq -\Phi^+$. We will denote by $P_J$ (resp., $P_J^{-}$) the parabolic subgroup of $G$ whose Lie algebra is direct sum of weight subspaces of the Lie algebra of $G$ of weights $\Phi^+\coprod \Phi^-_J\coprod \{0\}$ (resp., $-\Phi^+\coprod\Phi^+_J\coprod\{0\}$) with respect to the adjoint action of $S$. Let $L_J\subset P_J$ be the Levi subgroup containing $S$ such that its root system has $J$ as a basis, and let $Z_J$ be the center of $L_J$. Let $P_J^-$ be the opposite parabolic subgroup so that $P_J\bigcap P_J^-=L_J$.

For any $\alpha\in \Phi$, let $(\alpha)\coloneq\mathbb{Z}_{>0}\alpha\bigcap \Phi$. Then we have either $(\alpha)=\{\alpha\}$ or $(\alpha)=\{\alpha,2\alpha\}$. Let $\Phi_{nd}^{\pm}\subset\Phi^{\pm}$ be the subsets that consist of all non-divisible roots, namely
\[\Phi_{nd}^{\pm}\coloneq\{\alpha\in \Phi^{\pm}| \ \alpha/2\notin \Phi \}.\]
Then we have $\Phi^{\pm}=\bigcup_{\alpha\in \Phi_{nd}^{\pm}}(\alpha)$. For any $\alpha\in \Phi$, let $U_{(\alpha)}$ be the unipotent subgroup determined by $\alpha$, c.f., \cite[21.9 Proposition]{Borellinearalgebraicgroup}. 

We introduce the relative Weyl group $W\coloneq N_G(S)/Z_G(S)$. We remark that the group $W$ can be realized as a subquotient of the absolute Weyl group of $\widetilde{W}$, and that $W$ is a subgroup of $\widetilde{W}$ if $G$ is quasi-split, see \cite[\S~6.10]{boreltits}. For any $\alpha\in \Delta$, we denote by $s_{\alpha}$ the corresponding simple reflection in $W$. Let $l: W\rightarrow \mathbb{Z}$ be the length function, and let $w_0\in W$ be the longest element. Let $\leq$ denote the usual Bruhat order on $W$.

For a subset $J\subset \Delta$, let $W_J\subset W$ be the subgroup generated by the simple reflections given by $J$. Let $w_{0, J}$ be the longest element in $W_J$. For any $w\in W$ and $J\subset \Delta$, let $w^{J}$ be the unique element of smallest length in the coset $w\cdot W_J$ (\cite[\S~21.21]{Borellinearalgebraicgroup}). We denote by $W^J$ the set of $w^J$ ($w\in W$). Then we have 
\begin{equation}\label{equationofWJ}
    W^J=\{w\in W\vert w(J)\subset \Phi^+\}.
\end{equation}


\subsubsection{$\ast$-action}\label{staraction}
We recall the definition of the $\ast$-action defined by Borel and Tits in \cite[\S~6.2]{boreltits}. 

For an element $\gamma\in \Gal(k_s/k)$, $\gamma\cdot \widetilde{\Phi}^+$ is again a positive system of roots for $\widetilde{\Phi}$. By \cite[Theorem~8.2.8]{linearalgrpSpringer}, there is a unique $w_{\gamma}\in \widetilde{W}$ such that $w_{\gamma}\cdot\gamma\cdot \widetilde{\Phi}^+=\widetilde{\Phi}^+$. We define the $\ast$-action of $\Gal(k_s/k)$ on $\Delta$ by 
$$\gamma\ast r\coloneq w_{\gamma}\cdot \gamma\cdot r,\; r\in\Delta.$$

\begin{remark}
    An intrinsic way to express $\ast$-action is to introduce the notion of a scheme of Dynkin diagram, see \cite[Remark~7.1.2]{redctiveconrad} and \cite[exposé~XXIV, \S~3]{SGA3III}. 

    Since $Z_G(S)$ normalizes $U^+$, $\widetilde{W}_0$ stabilizes $\widetilde{\Phi}^+\backslash \widetilde{\Phi}_0$. By combining with the fact that $\gamma\cdot \widetilde{\Phi}_0=\widetilde{\Phi}_0$ and $\gamma\cdot (\widetilde{\Phi}^+\backslash \widetilde{\Phi}_0)=\widetilde{\Phi}^+\backslash \widetilde{\Phi}_0$, we have 
$w_{\gamma}\in \widetilde{W}_0.$

If $G$ is quasi-split over $k$, then the $\ast$-action is simply the natural action of $\Gal(k_s/k)$ on $\widetilde{\Delta}$. 
\end{remark}

\subsection{Split case}\label{sectionsplitcase} We now review the constructions and some properties of wonderful compactification for split reductive groups of adjoint type. 

\subsubsection{Construction}Let $G^{\text{sc}}\rightarrow G_{k_s}$ be the simply connected cover, and let $T^{\text{sc}}\subset G^{\text{sc}}$ be the preimage of $T_{k_s}$.
Given a regular dominant weight $\lambda\in X^{\ast}(T^{\text{sc}})$ (with respect to $B$), by \cite[Lemma~6.1.1 and Remark~6.1.2]{BrionKumar}, we can fix an algebraic finite-dimensional representation $V_{\lambda}$ of $G^{\text{sc}}$ such that 
\begin{itemize}
    \item the $\lambda$-weight subspace is of dimension 1;
    \item any weight of $T^{\text{sc}}$ other than $\lambda$ in the representation $V_{\lambda}$ is lower than $\lambda$;
    \item for any simple root $\alpha\in\widetilde{\Delta}$, the weight space of $\lambda-\alpha$ is nonzero.
\end{itemize}
By the adjointness of $G$ and the regularity of $\lambda$ (see \cite[Lemma~6.1.3]{BrionKumar}), the group $G_{k_s}$ is embedded into  $\mathbb{P}(\End(V_{\lambda}))$ in such a way that the following square commutes:
$$\xymatrix{
G_{\text{sc}} \ar[d] \ar[r]
&\End(V_{\lambda})\ar[d]\\
G_{k_s} \ar@{^{(}->}[r]         &\mathbb{P}(\End(V_{\lambda})).}$$
We define the wonderful compactification $\overline{G_{k_s}}$ of $G_{k_s}$ to be the closure of $G_{k_s}$ in $\mathbb{P}(\End(V_{\lambda}))$. The $(G_{k_s}\times G_{k_s})$-action on $G_{k_s}$ naturally extends to $\overline{G_{k_s}}$. By \cite[Theorem~6.1.8 (iv)]{BrionKumar}, the wonderful compactification $\overline{G_{k_s}}$ does not depend on the choice of $\lambda$ and $V_{\lambda}$. 

Two different constructions of the compactification $\overline{G_{k_s}}$ via Grassmannian varieties and via Hilbert schemes are given in \cite{completesymmetricvarieties} and \cite{compactificationHilbertsch} respectively.

\subsubsection{Properties of\; $\overline{G_{k_s}}$}\label{splitproperties}
The important features of the wonderful compactification $\overline{G_{k_s}}$ (see, for instance, \cite[\S~1.1]{intersectioncohomologyofBorbitclosures} \cite[6.1.8 Theorem]{BrionKumar}) which will be used in this paper are 
\begin{itemize}
    \item[(1)] The immersion $T_{k_s}\longrightarrow \overline{G_{k_s}}$ factors through a subvariety $\widetilde{T_{k_s}}\coloneq \prod_{\widetilde{\Delta}}\mathbb{A}_{1,k_s}$ of $\overline{G_{k_s}}$ and the morphism $T_{k_s}\longrightarrow \widetilde{T_{k_s}}$ is given by $\widetilde{\Delta}$. 
    \item[(2)] For a subset $\widetilde{I}\subset \widetilde{\Delta}$, let $z_{\widetilde{I}}\in \overline{T}$ be the point whose $i$-component is $0$ if $i\in \widetilde{I}$ and is $1$ if $i\notin \widetilde{I}$. Let $X_{\widetilde{I}}\coloneq (G_{k_s}\times_{k_s}G_{k_s})\cdot z_{\widetilde{I}}$. Then every $G_{k_s}\times_{k_s}G_{k_s}$-orbit of $\overline{G_{k_s}}$ is of this form. Moreover, we have 
    $$\Stab_{G_{k_s}\times_{k_s}G_{k_s}}(z_{\widetilde{I}})=(\mathscr{R}_u(P^-_{\widetilde{I}})\times_{k_s}\mathscr{R}_u(P_{\widetilde{I}}))\cdot (Z_{\widetilde{I}}\times\{1\})\cdot \Diag(L_{\widetilde{I}}).$$
    For $\widetilde{I}, \widetilde{J} \subset \widetilde{\Delta}$, we have  $X_{\widetilde{I}}\subset \overline{X_{\widetilde{J}}}$ if and only if $\widetilde{I}\subset \widetilde{J}$. 
    \item[(3)] Let $P_{\widetilde{I}}\subset G_{k_s}$ be the standard parabolic subgroup defined by $\widetilde{I}\subset\widetilde{\Delta}$, let $P_{\widetilde{I}}^-\supset B^-$ be the opposite parabolic subgroup so that $L_{\widetilde{I}}$ is a Levi subgroup of both $P_{\widetilde{I}}$ and $P_{\widetilde{I}}^-$ and let $Z_{\widetilde{I}}$ be the center of $L_{\widetilde{I}}$. Then the choice of the base point $z_{\widetilde{I}}$ of $X_{\widetilde{I}}$ in (2) gives an isomorphism
    $$X_{\widetilde{I}}\cong (G_{k_s}\times G_{k_s})\times^{P_{\widetilde{I}}^-\times P_{\widetilde{I}}} L_{\widetilde{I}}/Z_{\widetilde{I}}.$$
    This isomorphism further extends to an isomorphism
    $$\overline{X_{\widetilde{I}}}\cong (G_{k_s}\times G_{k_s})\times^{P_{\widetilde{I}}^-\times P_{\widetilde{I}}} \overline{L_{\widetilde{I}}/Z_{\widetilde{I}}},$$
    where the first bar indicates taking closure in $\overline{G}$ and $\overline{L_{\widetilde{I}}/Z_{\widetilde{I}}}$ is the wonderful compactification of the adjoint reductive group $L_{\widetilde{I}}/Z_{\widetilde{I}}$.
\end{itemize}

\subsection{Galois descent}\label{galoisdescent}
We fix a split reductive group $G_0$ over $k$ and an isomorphism $G_{k_s}\cong(G_0)_{k_s}$ of $k_s$-groups.
Let $\overline{(G_0)_{k_s}}$ be the wonderful compactification of $(G_0)_{k_s}$ which we recall in \Cref{sectionsplitcase}. 
By, for instance, \cite[11.3.3~Proposition]{linearalgrpSpringer}, the isomorphism we fixed gives rise to a $1$-cocycle 
$$c\colon \Gamma\longrightarrow \Aut_{k_s-\Grp}((G_0)_{k_s})\in Z^1(\Gamma, \Aut_{k_s-\Grp}((G_0)_{k_s})).$$
By \cite[exposé~XXIV, théorème~1.3]{SGA3III}, we have the decomposition
$$\Aut_{k_s-\Grp}((G_0)_{k_s})\cong G_0(k_s) \rtimes \Out((G_0)_{k_s}),$$
where $\Out((G_0)_{k_s})$ is the outer automorphism group.

By \cite[\S~6.2]{li2023equivariant}, the $1$-cocycle $c$ extends to a (necessarily unique) $1$-cocycle 
$$\overline{c}\colon \Gamma\longrightarrow \Aut_{k_s}(\overline{(G_0)_{k_s}})\in Z^1(\Gamma, \Aut_{k_s}(\overline{(G_0)_{k_s}})).$$ 
See also \cite[\S~5.5]{rationalpointsoncompactification} for another description of this extension via representation theory in characteristic zero.
Since $\overline{(G_0)_{k_s}}$ is (quasi-)projective over $k_s$, it descends to a smooth projective $k$-variety $\overline{G}$ containing $G$ as an open dense subvariety.

\section{Relative Bruhat decomposition}\label{sectionbruhatdecomposition}
We shall keep the notations of \Cref{sectionwonderfulcompactification}.
The main goal of this section is to decompose $\overline{G}(k)$ into disjoint union of $P(k)\times P^-(k)$-orbits. We first decompose $\overline{G}(k)$ into disjoint union of $G(k)\times G(k)$-orbits in \Cref{propositionofbasepoint}. Then we decompose each $G(k)\times G(k)$-orbit in \Cref{borbitproposition}.

\subsection{$G(k)\times G(k)$-orbits}
We begin with the following lemma which is essentially a generalization of \cite[Lemma~2.2] {subvarietiesofgroupcompactification} with the special attention to rational points and general base field, whereas \emph{loc. cit.} is stated for an algebraically closed base field.

\begin{lemma}\label{lemmaofvaluation}
    For a closed subvariety $Z\subset G$, let $\overline{Z}$ be the closure of $Z$ in $\overline{G}$. Then, a point $x\in \overline{G}(k)$ belongs to $\overline{Z}(k)$ if and only if there exist a finite field extension $k'/k$ and a point $\theta\in Z(k'((t)))$ such that the unique extension $\theta'\in \overline{G}(k'[[t]])$ of $\theta$ specializes to $x$. For the “only if” part, if $\overline{Z}$ is further assumed to be $k$-smooth at $x$, then $k'$ could be $k$.
\end{lemma}

\begin{proof}
    For the “if” direction, since $\overline{Z}$ is also proper over $k$, the unique extension $\theta'$ must lie in $\overline{Z}(k[[t]])$. Hence the specialization $x$ belongs to $\overline{Z}(k)$.

    Conversely, suppose that $x\in \overline{Z}(k)$. By Bertini theorem \cite[Proposition~4.1.3]{CKproblemtorsors} (which works without further restriction on the base field $k$), we can find $\dim(Z)-1$ hypersurfaces $H_i\subset \overline{Z}$ for $i=1,...,\dim(Z)-1$ such that (among other things)
    \begin{itemize}
         \item the intersection $C\coloneq \bigcap_{i=1}^{\dim(Z)-1}$ contains $x$;
        \item the intersection $C$ is of pure dimension $1$ and $C\bigcap (\overline{Z}\backslash Z)$ is of dimension zero;
        \item if $\overline{Z}$ is $k$-smooth at $x$, then so is $C$.
    \end{itemize}
    We take $C'$ to be the normalization of an irreducible component of $C$ which contains $x$. Then $C'$ is a projective regular curve over $k$ (\cite[Proposition~6.40, Proposition~12.43]{GortzWedhorn}). We choose a point $x'\in C'$ lying over $x$, and take its complete local ring $\mathfrak{o}$. By, for instance, \cite[Chapter~II, \S~4, Theorem~2]{localfieldsserre}, $\mathfrak{o}$ is isomorphic to the ring $k'[[t]]$ of formal power series where $k'$ is the residue field of $x'$ which is finite over $k$ (\cite[0C45(5)]{stacks-project}). 
    Then, the natural morphism $C'\rightarrow \overline{Z}$ will give rise to the desired section $\theta$. 
    
    When $C$ is $k$-smooth at $x$, then $x$ lies inside the normal locus of the chosen irreducible component of $C$ containing it (\cite[033N]{stacks-project}). Then the residue field $k'$ equals to $k$ because the normalization morphism becomes an open immersion after restricting to an open neighbourhood of $x'$ in $C'$ (\cite[Remark~12.46]{GortzWedhorn}).
\end{proof}

Let $\overline{T_{k_{s}}}\coloneq \bigcup_{w\in \widetilde{W}} (w,w)\cdot \widetilde{T_{k_s}}\subset \overline{G_{k_s}}$ which is the closure of $T_{k_s}$ in $\overline{G_{k_s}}$. Then $\overline{T_{k_{s}}}$ is a smooth toric variety (for $T_{k_s}$) corresponding to the fan of Weyl chambers and their faces, c.f., \cite[Lemma~6.1.6(ii)]{BrionKumar}. We are going to find a subvariety of $\overline{T_{k_{s}}}$ which gives us "base points" of $G(k)\times G(k)$-orbits of $\overline{G}(k)$, just like in the split case, c.f., \Cref{splitproperties} (2). For this, we need the following combinatorial lemma.

\begin{lemma}\label{lemmaoffans}
For $\gamma\in\Gamma$,
    let \[C\coloneq \{\lambda\in X_{\ast}(T)_{\mathbb{R}}\vert \langle \lambda, \alpha\rangle \geq 0 \; \text{for any}\; \alpha\in \widetilde{\Delta}\}\text{ and }C_{\gamma}\coloneq \{\lambda\in X_{\ast}(T)_{\mathbb{R}}\vert \langle \lambda, \gamma\cdot \alpha\rangle \geq 0 \; \text{for any}\; \alpha\in \widetilde{\Delta}\}.\] Then we have
    \begin{equation}\label{equationoffans}
        \begin{aligned}       
           &\{\lambda\in X_{\ast}(T)_{\mathbb{R}}\vert \langle \lambda, \alpha\rangle \geq 0 \; \text{for any}\; \alpha\in \widetilde{\Delta}\backslash \widetilde{\Delta}_0;\;\langle \lambda, \alpha\rangle = 0 \; \text{for any}\; \alpha\in \widetilde{\Delta}_0\}\\
        =&\{\lambda\in X_{\ast}(T)_{\mathbb{R}}\vert \langle \lambda, \gamma\cdot\alpha\rangle \geq 0 \; \text{for any}\; \alpha\in \widetilde{\Delta}\backslash \widetilde{\Delta}_0;\;\langle \lambda, \gamma\cdot\alpha\rangle = 0 \; \text{for any}\; \alpha\in \widetilde{\Delta}_0\},   
    \end{aligned}  
    \end{equation}
    which is, by definition, a face of both $C$ and $C_{\gamma}$.
\end{lemma}

\begin{proof}
    Since $\widetilde{\Phi}_0$ is the root system of $Z_G(S)_{k_s}$ with respect to $T_{k_s}$ having $\widetilde{\Delta}_0$ as a set of simple roots (\cite[Proposition~15.5.3~(ii)]{linearalgrpSpringer}),
    for $\alpha\in \widetilde{\Delta}_0$, we have that $\gamma\cdot \alpha$ lies in $\widetilde{\Phi}_0$. 
    On the other hand, if $\alpha\in \widetilde{\Delta}\backslash \widetilde{\Delta}_0$, then, by \cite[Corollary~15.5.4]{linearalgrpSpringer}, we have 
    $$\gamma\cdot \alpha= \gamma\ast \alpha+ \sum_{\eta\in\widetilde{\Delta}_0}n_{\eta}\eta,\; n_{\eta}\in \mathbb{Z}_{\textgreater 0}.$$
   Moreover, according to \cite[Proposition~15.5.3~(i)]{linearalgrpSpringer}, the $\ast$-action stabilizes $\widetilde{\Delta}\backslash \widetilde{\Delta}_0$.
   Hence one inclusion of \Cref{equationoffans} follows. For the other inclusion, it suffices to use the known inclusion by replacing $\widetilde{\Delta}$ by $\gamma^{-1}\cdot \widetilde{\Delta}$.
\end{proof}

We denote the cone appears in \Cref{lemmaoffans} by 
\[C_0\coloneq \{\lambda\in X_{\ast}(T)_{\mathbb{R}}\vert \langle \lambda, \alpha\rangle \geq 0 \; \text{for any}\; \alpha\in \widetilde{\Delta}\backslash \widetilde{\Delta}_0;\;\langle \lambda, \alpha\rangle = 0 \; \text{for any}\; \alpha\in \widetilde{\Delta}_0\}.\] 
Then by the theory of toric varieties (for instance \cite[Proposition~1.3]{toric}), the open subvariety \[\prod_{\widetilde{\Delta}\backslash \widetilde{\Delta}_0} \mathbb{A}_{1, k_s} \times_{k_s}\prod_{\widetilde{\Delta}_0}\mathbb{G}_{m, k_s} \subset \widetilde{T_{k_s}}\] is precisely the toric variety corresponds to $C_0$. By \Cref{lemmaoffans}, we have $\gamma\cdot C_0=C_0$ for any $\gamma\in \Gamma$. Therefore, $\prod_{\widetilde{\Delta}\backslash \widetilde{\Delta}_0} \mathbb{A}_{1, k_s} \times_{k_s}\prod_{\widetilde{\Delta}_0}\mathbb{G}_{m, k_s}\subset \overline{T_{k_s}}$ is a subvariety that is stable under Galois action.

For any $\gamma\in \Gamma$, $\prod_{\widetilde{\Delta}\backslash \widetilde{\Delta}_0} \mathbb{A}_{1, k_s} \times_{k_s}\prod_{\widetilde{\Delta}_0}\mathbb{G}_{m, k_s}$ is embedded, as an open subvariety, into $\prod_{\gamma\cdot \widetilde{\Delta}} \mathbb{A}_{1,k_s}$ by sending (for each $r\in \widetilde{\Delta}$) the $r$-component to the $\gamma\cdot r$-component. This embedding corresponds to the inclusion $C_{\gamma}\subset \gamma\cdot C_0$. For any $r\in \widetilde{\Delta}$ and $\gamma \in \Gamma$, let $\mathbb{X}_{\gamma\cdot r}\in k_s[\gamma \cdot \widetilde{T_{k_s}}]$ be the coordinate corresponding to $\gamma\cdot r$. Then for any $\gamma\in \Gamma$ and $r\in \widetilde{\Delta}$, $\mathbb{X}_{\gamma\cdot r}$ is a well-defined regular function on $\prod_{\widetilde{\Delta}\backslash \widetilde{\Delta}_0} \mathbb{A}_{1, k_s} \times_{k_s}\prod_{\widetilde{\Delta}_0}\mathbb{G}_{m, k_s}$.


We define the closed subvariety of  $\prod_{\widetilde{\Delta}\backslash \widetilde{\Delta}_0} \mathbb{A}_{1, k_s} \times_{k_s}\prod_{\widetilde{\Delta}_0}\mathbb{G}_{m, k_s}$ by
$$\widetilde{S_{k_s}}\coloneq \Big(\bigcap_{r\in\widetilde{\Delta}_0} V(\mathbb{X}_r-1)\Big)\bigcap \Big(\bigcap_{r\in \widetilde{\Delta}\backslash\widetilde{\Delta}_0}V(\{\mathbb{X}_r-\mathbb{X}_{\sigma\cdot r}\}_{\sigma\in \Gamma})\Big)\subset \prod_{\widetilde{\Delta}\backslash \widetilde{\Delta}_0} \mathbb{A}_{1, k_s} \times_{k_s}\prod_{\widetilde{\Delta}_0}\mathbb{G}_{m, k_s},$$
where $V$ indicates taking the zero locus of the corresponding ideal in $\prod_{\widetilde{\Delta}\backslash \widetilde{\Delta}_0} \mathbb{A}_{1, k_s} \times_{k_s}\prod_{\widetilde{\Delta}_0}\mathbb{G}_{m, k_s}$.
Because the defining ideal is stable under the action of $\Gamma$, $\widetilde{S_{k_s}}$ is stable under Galois action as well.

Then $\widetilde{S_{k_s}}$ descends to a $k$-subvariety $\widetilde{S}\subset \overline{G}$. Since $G$ is adjoint, we have $\mathbb{Z}\Delta=X^{\ast}(S)$. Moreover, since $S_{k_s}$ is dense in $\widetilde{S_{k_s}}$, the embedding $S_{k_s}\hookrightarrow\widetilde{S_{k_s}}$ induced from \Cref{splitproperties}~(1) also descends to a $k$-immersion which fits into the following diagram
$$\xymatrix{
S\ar@{^{(}->}[dd]_{(\alpha)_{\alpha\in\Delta}}  \ar@{^{(}->}[drr]&  &\\
&  &  \overline{G}\\
\widetilde{S}\cong \prod_{\Delta} \mathbb{A}_{1,k}  \ar@{^{(}->}[urr] & &
}$$
where the other two arrows are given by the inclusion $G\subset \overline{G}$ and \Cref{splitproperties} (1). Now we show that $\widetilde{S}$ contains "base points" of $G(k)\times G(k)$-orbits of $\overline{G}(k)$.
The following result is a generalization of \cite[Proposition~A1] {behaviouratinfinityofbruhat} over general base field.

\begin{proposition}\label{propositionofbasepoint}
    For a $G(k)\times G(k)$-orbit $O$ in $\overline{G}(k)$, there are a unique point $b\in O$ and a unique subset $I\subset \Delta$ such that 
    \begin{itemize}
       \item[(1)] $b=\lim_{x\to 0} \lambda(x)\in \widetilde{S}(k)$ for some $\lambda\in X_{\ast}(S)$;
        \item[(2)] $(P\times_k P^-)\cdot b$ is open dense in $(G\times_k G)\cdot b$.
    \end{itemize}
    Moreover, the point $b$ is invariant under the actions of $\Diag(L_I)$ and of $Z_I\times_k Z_I$.
\end{proposition}

\begin{proof}
    Since $b$ is required to be a limit of a character of $S$, viewed as a point in $\overline{T}(k_s)$, it coincides with a base point in \Cref{sectionsplitcase} (2). Hence the uniqueness follows.

    Let us prove the existence of $b$.
    Take a point $o\in  O$. By \Cref{lemmaofvaluation}, there is a point $x\in G(k((t)))$ such that the unique lifting of $x$ in $\overline{G}(k[[t]])$ specializes to $o$. 
    By Cartan decomposition (see, for instance, \cite[Theorem~6.2.2]{genericallytrivialtorsorsconstant2025}), i.e., $G(k((t)))= \coprod_{\lambda\in X_{\ast}(S)/N_G(S)(k)} G(k[[t]])\cdot\lambda(t)\cdot G(k[[t]])$, we can find $g_1(t), g_2(t)\in G(k[[t]])$ and $\lambda\in X_{\ast}(S)$ such that   
       \begin{equation}\label{equationCartandecomposition}
            x=g_1(t)\cdot \lambda(t)\cdot g_2(t);
        \end{equation}  
        \begin{equation}\label{eqcondition}
          \text{and} \; \langle\lambda, \alpha\rangle\geq 0 \;\text{for any relative simple root }\;\alpha\in \Delta.
        \end{equation}
        
    If we let $b\coloneq \lim_{t\to 0} \lambda(t)$, taking limit at $0$ for \Cref{equationCartandecomposition} shows that $b\in O$. Moreover, since the embedding $S\hookrightarrow \widetilde{S}$ is given by $\Delta$, the fact that the limit $b$ exists in $\widetilde{S}(k)$ is equivalent to \Cref{eqcondition}.

    The assertion (2) can be shown by \Cref{sectionsplitcase} (3) and Galois descent. Here we give a direct proof by following the proof of \cite[Proposition~A1]{behaviouratinfinityofbruhat}. Recall the following dynamic description
    $$U(\lambda)\coloneq \{g\in G\vert \lim_{t\to 0} \lambda(t)\cdot  g \cdot \lambda(t)^{-1}=1 \},$$
    $$L(\lambda)\coloneq\{g\in G\vert \lambda(t)\cdot  g \cdot \lambda(t)^{-1}=g\}.$$
    For $g_1\in U(\lambda), g_2\in U(-\lambda)$, by taking limit of $t$ to $0$ for the following two equalities
    $$\lambda(t)=(\lambda(t)g_1\lambda(t)^{-1})\cdot\lambda(t)\cdot g_1^{-1},\;\;\lambda(t)=g_2\cdot \lambda(t)\cdot (\lambda(t)^{-1}g_2\lambda(t))^{-1},$$
    it follows that $U(-\lambda)\times_k U(\lambda)\subset \Stab_{G\times_k G}(b)$. Note that, by the definition of $L(\lambda)$, we also have $\Diag(L(\lambda))\subset\Stab_{G\times_k G}(b)$. Hence, since $U(\lambda)\times_k L(\lambda)\times_k U(-\lambda)$ is an open subvariety of $G$ (see, for instance, \cite[Theorem~4.1.7~4]{redctiveconrad}),
    the image of the open subvariety 
    $$(P\times_k P^{-})\cdot (U(-\lambda)\times_k U(\lambda))\cdot \Diag(L(\lambda))\subset G\times_k G$$ 
    in $(G\times_k G)\cdot b$, which is equal to $(P\times_k P^-)\cdot b$, is also an open subvariety.
\end{proof}

\Cref{propositionofbasepoint} shows that the $G(k)\times G(k)$-orbits in $\overline{G}(k)$ is parametrized by the set of subsets of $\Delta$. For a subset $I\subset \Delta$, let 
$$b_I\in \widetilde{S}(k)$$ 
whose $i$-component is $0$ if $i\in I$ and is $1$ if $i\notin I$, and let $Y_I\coloneq (G(k)\times G(k)) \cdot b_I$.

As a corollary of \Cref{propositionofbasepoint}, we have the following interesting observation.

\begin{corollary}
    If $G$ is anisotropic, i.e., $G$ does not contain any nontrivial $k$-split torus, then $\overline{G}(k)=G(k)$.
\end{corollary}

\subsection{$P(k)\times P^-(k)$-orbits}
Now our goal is to further decompose $Y_I$ into disjoint union of $P(k)\times P^-(k)$-orbits. The strategy of the proof of the following \Cref{borbitproposition} is from \cite[\S~2.1]{behaviouratinfinityofbruhat}.


\begin{proposition}\label{borbitproposition}
    We now consider a $P(k)\times P^-(k)$-orbit $O$ contained in $Y_I$. Then there exist unique $\sigma, \tau\in W^I$ and $\rho\in W_I$ such that 
    $$O=(P(k)\times P^-(k)) (\sigma \rho,\tau)\cdot b_I. \footnote{Since $\sigma ,\rho$ and $ \tau$ are elements in the relative Weyl group $W$, the expression should be understood as the group actions by representatives of these elements in $N_G(S)$. Since $Z_G(S)\subset P\bigcap P^-$, the expression does not depend on the choice of representatives. Similar convention applies to later discussion.}  $$
\end{proposition}

\begin{proof}
We denote the the natural projection morphism by
    $$\pi_I: (G\times_k G)\times_k^{P_I^-\times_k P_I} \overline{L_I/Z_I}\longrightarrow G/P_I^-\times_k G/P_I.$$
By \Cref{splitproperties} (3) and Galois descent from \Cref{galoisdescent}, 
the wonderful compactification $\overline{L_I/Z_I}$ is identified with a fiber of $\pi_I$ via
\begin{equation}\label{equationoffiber}
    \overline{L_I/Z_I}\cong (P_I^-\times_k P_I)\times_k^{P_I^-\times_k P_I} \overline{L_I/Z_I}\longhookrightarrow (G\times_k G)\times_k^{P_I^-\times_k P_I} \overline{L_I/Z_I}.
\end{equation}

Since $\pi_I$ is $(G\times_k G)$-equivariant, the image $\pi_I(O)\subset (G/P_I^-)(k)\times (G/P_I)(k)$ is still a $P(k)\times P^-(k)$-orbit. Moreover, since the natural maps $G(k)\rightarrow (G/P_I^-)(k)$ and $G(k)\rightarrow (G/P_I)(k)$ are surjective (\cite[Proposition~20.5]{Borellinearalgebraicgroup}), by the relative Bruhat decomposition (\cite[\S~21.16]{Borellinearalgebraicgroup})
$$G(k)/P_I(k)= \coprod_{w\in  W/W_I} P(k)\cdot w\cdot P_I(k)/P_I(k),$$
there exist unique $\sigma, \tau\in W^I$ such that 
$$\pi_I(O)=(P(k)\cdot \sigma\cdot P_I^{-}(k)/P_I^{-}(k))\times (P(k)^{-}\cdot \tau\cdot P_I(k)/P_I(k)).$$

Note that $\pi_I^{-1}(\sigma\cdot P_I^{-}(k)/P_I^{-}(k),\tau\cdot P_I(k)/P_I(k))\bigcap O$ is stable under the action of 
\[(P(k)\times P^-(k))\bigcap (\sigma\cdot P_I^{-}(k)\cdot \sigma^{-1}\times\tau\cdot P_I(k)\cdot \tau^{-1} ),\] which contains $\sigma\cdot (L_I(k)\bigcap P(k))\cdot \sigma^{-1}\times\tau\cdot (L_I(k)\bigcap P^-(k))\cdot \tau^{-1}$ because $\sigma, \tau\in W^I$. Let
$$F\coloneq (\sigma^{-1}, \tau^{-1})\cdot (\pi_I^{-1}(\sigma\cdot P_I^{-}(k)/P_I^{-}(k),\tau\cdot P_I(k)/P_I(k))\bigcap O)\subset Y_I.$$
Then $F$ is contained in $\overline{L_I/Z_I}(k)$ and is stable under $L_I(k)\bigcap P(k)\times L_I(k)\bigcap P^-(k)$. Moreover, we have the equality
$$O=(P(k)\times P^-(k))\cdot (\sigma, \tau)\cdot F.$$
Now we are left to work out the set $F$.
By \Cref{propositionofbasepoint}, the stabilizer of $b_I$ contains $\Diag(L_I)$. This means that $F\subset L_I/Z_I$ under \Cref{equationoffiber}. Again, by the relative Bruhat decomposition for groups \emph{loc. cit.,} there exists a unique $\rho\in W_I$ such that 
$$F= (L_I(k)\bigcap P(k))\cdot\rho \cdot(L_I(k)\bigcap P^-(k)).$$
Since $\sigma, \tau\in W^I$, we thus finish our proof by observing 
$$P(k)\sigma(P(k)\bigcap L_I(k))= P(k)\sigma\;\; \text{and}\;\;P^-(k)\tau(P^-(k)\bigcap L_I(k))= P^-(k)\tau.$$
\end{proof}

\section{Geometric refinement}

In this section, we shall keep the notations of \Cref{sectionbruhatdecomposition}. We consider two subvarieties of $\overline{G}$
$$O_i\coloneq (P\times_k P^{-})\cdot (\sigma_i \rho_i,\tau_i)\cdot b_{I_i}\subset \overline{G},\;\text{with}\;\rho_i\in W_{I_i}\;,\sigma_i,\tau_i\in W^{I_i}\;\text{and}\; i={1,2},$$
which should be viewed as the "geometrizations" of $O_1(k)$ and $O_2(k)$.

\begin{theorem}\label{theoremofgeometrization}    
   If $O_1\bigcap O_2\neq \emptyset$, then $I_1=I_2$, $(\sigma_1,\tau_1)=(\sigma_2,\tau_2)$ and $\rho_1=\rho_2$.
\end{theorem}

\begin{proof}
    First of all, by \Cref{sectionsplitcase} (2), we must have $I_1=I_2$. For simplicity, we write $I$ for both $I_1$ and~$I_2$.
    Recall that we have the natural projection morphism
    $$\pi_I: (G\times_k G)\times_k^{P_I^-\times_k P_I} \overline{L_I/Z_I}\longrightarrow G/P_I^-\times_k G/P_I.$$
    Since 
    $\pi_I$ is $G\times_k G$-equivariant and $\rho_i\in W_I$, we have
    \begin{align*}
        \pi_I(O_i) & =(P\cdot \sigma_i \cdot P_I^-/P_I^- )\times_k (P^-\cdot \tau_i \cdot P_I/P_I)\\
        & =  (w_0\cdot P^- w_0^{-1}\cdot \sigma_i \cdot P_I^-/P_I^- )\times_k (w_0^{-1}\cdot P \cdot w_0 \cdot \tau_i \cdot P_I/P_I),\;\text{for}\; i=1,2,
    \end{align*}
    where two minimal parabolic subgroups $P$ and $P^-$ are conjugate to each other by $w_0\in W$ (\cite[théorème~4.13]{boreltits}).
    Since $\pi_I(O_1\bigcap O_2)\neq \emptyset$ and $\sigma_i, \tau_i \in W^I$, by \cite[Proposition~21.29(iv)]{Borellinearalgebraicgroup} (which is a variant of \Cref{RelativeBruhatforgroup} (2)), we conclude that $\sigma_1=\sigma_2$ and $\tau_1=\tau_2$. For simplicity, we write $\sigma$ (resp., $\tau$) for $\sigma_1$ and $\sigma_2$ (resp., $\tau_1$ and $\tau_2$).

    Now we are left to show that $\rho_1=\rho_2$. For this, let $p_i\in P(\overline{k}), q_i\in P^{-}(\overline{k})$ for $i=1,2$ such that 
    $$(p_1, q_1)\cdot (\dot{\sigma}\dot{\rho_1},\dot{\tau})\cdot  b_I= (p_2,q_2)\cdot (\dot{\sigma}\dot{\rho_2}, \dot{\tau})\cdot b_I \in O_1\bigcap O_2,$$
    where $\dot{\sigma},\dot{\rho_1},\dot{\rho_2},\dot{\tau}$ are some representatives of $\sigma,\rho_1,\rho_2,\tau$ in $N_G(S)(k)$ respectively.
    By \Cref{splitproperties} (2) and Galois descent, this amounts to that $(\dot{\rho}_2^{-1} \cdot \dot{\sigma}^{-1}\cdot p_2^{-1}\cdot p_1\cdot \dot{\sigma} \cdot \dot{\rho}_1,\; \dot{\tau}^{-1}\cdot q_2^{-1}\cdot  q_1\cdot \dot{\tau})$ belongs to
    $$\Stab_{G\times_k G}(b_I)=\{(u\cdot l,l'\cdot u')\in ((\mathscr{R}_u(P^-_I)\rtimes L_I)\times_k (L_I \ltimes \mathscr{R}_u(P_I))\vert l\cdot l'^{-1}\in Z_I\}.$$
   On the other hand, since $\tau\in W^I$ (\Cref{equationofWJ}), 
    the Levi component of 
    $$\dot{\tau}^{-1}\cdot q_2^{-1}\cdot  q_1\cdot \dot{\tau}\in P^-_I=L_I\ltimes \mathscr{R}_u(P^-_I)\;\; $$ 
    lies in $P^-\bigcap L_I$ (resp., $P\bigcap L_I$).
    Thus, the $L_I$-component of $\dot{\rho}_2^{-1} \cdot \dot{\sigma}^{-1}\cdot p_2^{-1}\cdot p_1\cdot \dot{\sigma} \cdot \dot{\rho}_1$ also lies in $P^-$. This means that
    $\rho_1\in P^-\cdot \rho_2\cdot P.$
    Finally we conclude that $\rho_1=\rho_2$ thanks to \Cref{RelativeBruhatforgroup} (2).
\end{proof}

For an algebraic variety $H$ on which an algebraic group $A$ acts, consider a $k$-point $h\in H(k)$ and its orbit $A\cdot h$. In general, it is not true that $(A\cdot h)(k)=A(k)\cdot h $. For instance, we can consider the action $\mathbb{G}_{m,k}\times_k\mathbb{G}_{m,k}\rightarrow\mathbb{G}_{m,k}$ by $(x,y)\mapsto x^2y$ when $(k^{\times})^{2}\neq k^{\times}$. However, in our situation, we have the following result.

\begin{corollary}\label{kptsoforbit}
    Let $I\subset \Delta$, $\sigma, \tau\in W^I$ and $\rho\in W_I$. Consider the subvariety
    $O=(P\times P^-)\cdot (\sigma \rho,\tau)\cdot b_I \subset \overline{G}$. Then we have $O(k)=P(k)\times P^-(k)\cdot (\sigma \rho,\tau)\cdot b_I$.
\end{corollary}

\begin{proof}
    The statement that $O(k)$ contains more than one $P(k)\times P^-(k)$-orbits would contradict \Cref{theoremofgeometrization} by applying \Cref{borbitproposition}.
\end{proof}

\section{Relative Bruhat order}\label{sectionbruhatorder}

In the present section, we consider the Zariski closure relation between $P\times_kP^{-}$-orbits in $\overline{G}$. We will keep the notations of \Cref{sectionbruhatdecomposition}. Our main theorem will be \Cref{relbruhatmainthm} which will be the combination of \Cref{closureinsamegorbit} and \Cref{irrecompofintersection}. We start by proving some preparatory results about the structure of $(P\times_kP^-)$-orbits closures.

In this section, unless mentioned otherwise, for a subvariety $Y\subset \overline{G}$ we will use $\overline{Y}$ to denote the Zariski closure of $Y$ in $\overline{G}$.

\subsection{Action of parabolic subgroups}
In this part, we show some technical results about the actions of certain parabolic subgroups, which will play an important role in the proof of \Cref{relbruhatmainthm}.

We proceed by describing the structure of a $(P\times_kP^{-})$-orbit. Recall that a Bruhat cell $PwP$ is isomorphic to the direct product of $P$ and unipotent subgroup determined by $w$ (\cite[21.14~Lemma]{Borellinearalgebraicgroup}). More precisely, for any $w\in W$, let 
\[U_w\coloneq\mathscr{R}_u(P)\bigcap w\mathscr{R}_u(P)w^{-1}= \prod_{\alpha\in \Phi^+_{nd}\bigcap w\Phi^+ } U_{(\alpha)}, \]
\[U^-_{w}\coloneq\mathscr{R}_u(P^-)\bigcap w\mathscr{R}_u(P^-)w^{-1}= \prod_{\alpha\in \Phi^-_{nd}\bigcap w\Phi^- } U_{(\alpha)}, \]
where the products can be taken  with respect to any order on the sets $\Phi^{\pm}_{nd}\bigcap w\Phi^{\pm}.$
Then for any representative $\dot{w}$ of $w$, the action maps
\[U_w^{\pm}\times_k P^{\mp}\to P^{\pm}wP^{\mp},\  (x,y)\mapsto x\dot{w}y\]
are isomorphisms. This fact suggests the following lemma.

\begin{lemma}\label{structureoforbit}
    Let $J\subset \Delta$ be a subset of simple roots and let 
    \[O=(P\times_kP^{-})\cdot(\sigma\rho,\tau)\cdot b_J\]
    be the $(P\times_kP^-)$-orbit corresponds to $\sigma,\tau\in W^J$ and $\rho\in W_J$. Let $\dot{\sigma}$ (resp. $\dot{\rho}$, $\dot{\tau}$) be a representative of $\sigma$ (resp. $\rho$, $\tau$) in $N_G(S)(k)$.
    Then the action map 

    \[U_{\sigma\rho}\times_k \Big( P^-  \bigcap (\tau P_J^-\tau^{-1})\Big) \to O 
    \]
    \[(x,y)\mapsto (x,y)\cdot(\dot{\sigma}\dot{\rho},\dot{\tau})\cdot b_J\]
    induces an isomorphism  $U_{\sigma\rho}\times_k (( P^-  \bigcap (\tau P_J^-\tau^{-1}))/Z_J) \simeq O  $. Similarly, the action map 
     \[ \Big( P  \bigcap (\tau P_J\tau^{-1})\Big) \times_k U_{\tau\rho^{-1}}^- \to O 
    \]
     \[(x,y)\mapsto (x,y)\cdot(\dot{\sigma}\dot{\rho},\dot{\tau})\cdot b_J\]
    induces an isomorphism $(( P  \bigcap (\tau P_J\tau^{-1}))/Z_J) \times_k U_{\tau\rho^{-1}}^- \simeq O $.
\end{lemma}

\begin{proof}
    By the Bruhat decomposition for $(G\times_k G)/(P_J^-\times_kP_J)$, the map sending $(x,y)$ to $(x\dot{\sigma},y\dot{\tau})(P_J^-\times_kP_J)$ induces an isomorphism 
    \[(\prod_{\alpha }U_{(\alpha)})\times_k(\prod_{\beta} U_{(\beta)})\simeq (P\times_k P^-)\cdot(\dot{\sigma},\dot{\tau})\cdot(P_J^-\times_kP_J)/(P_J^-\times_kP_J), \]
    where the products are taken over $\alpha \in \Phi^+_{nd}\bigcap\sigma (\Phi^+-\Phi_J^+ )$ and $\beta\in \Phi^-_{nd}\bigcap \tau(\Phi^--\Phi_J^-) $ respectively.

    The Bruhat decomposition for $L_J/Z_J$ implies the action map
    \[\Big(U_{\rho}\bigcap L_J\Big)\times_k \Big(P^-\bigcap L_J\Big) \to \Big(P\bigcap L_J \Big)\cdot \dot{\rho} \cdot \Big(P^-\bigcap L_J\Big)/Z_J, \]
    which sends $(x,y)$ to $x\dot{\rho} y^{-1}$, has kernel $\{1\}\times_k Z_J$. Because $\sigma$ is a minimal length representative of $W/W_J$, we have $\sigma\Phi_J^+\subset \Phi^+$. Moreover, because $\rho(\Phi^+-\Phi^+_J)=(\Phi^+-\Phi^+_J)$, we have 
    \[\Phi^+\bigcap (\sigma\rho\Phi^+)=(\Phi^+\bigcap (\sigma(\Phi^+-\Phi_J^+))\bigcup \sigma(\Phi_{J}^+\bigcap \rho(\Phi^+_J))\]
    which implies 
    \[U_{\sigma\rho}=\Big(\prod_{\alpha \in \Phi^+_{nd}\bigcap\sigma (\Phi^+-\Phi_J^+ )} U_{(\alpha)}\Big) \dot{\sigma}\Big(U_{\rho}\bigcap L_J \Big)\dot{\sigma}^{-1}.\]
     Hence, combining these two statements implies the first case of \Cref{structureoforbit}. The opposite case can be proved similarly by observing that $(\dot{\sigma}\dot{\rho},\dot{\tau})\cdot b_J=(\dot{\sigma},\dot{\tau}\dot{\rho}^{-1})\cdot b_J$.
\end{proof}

For any $w\in W$, we  write $d(w)\coloneq\mathrm{codim}_G(P\cdot w\cdot P^-)=\dim(P\cdot w\cdot P/P)$. Then we have
\begin{itemize}
    \item $d(w)=\sum_{\alpha\in\Phi^+_{nd}\bigcap w(\Phi^-)}\dim(U_{(\alpha)})$
    \item $d(w_1)+d(w_2)=d(w_1w_2)$ when $l(w_1)+l(w_2)=l(w_1w_2)$
    \item For any $I\subset \Delta$ and $w\in W_I$, we have
    \[d(w)=\dim((P\bigcap L_I)\cdot w\cdot (P\bigcap L_I)/(P\bigcap L_I))=\mathrm{codim}_{L_I}((P\bigcap L_I)\cdot w\cdot (P^-\bigcap L_I))\]
    \item For any $I\subset \Delta$ and $w\in W^I$, we have
    \[d(w)=\mathrm{codim}_{G}(P\cdot w\cdot P^-_I)=\mathrm{codim}_G(P^-\cdot w\cdot P_I)\]
\end{itemize}
which follows from description of Bruhat cells, c.f., \cite[Proposition~21.22]{Borellinearalgebraicgroup}. Note that when $G$ is split over $k$ then $d(w)=l(w)$. Then, from Proposition \ref{structureoforbit}, we can deduce the following dimension formula.

\begin{corollary}\label{dimcomputation}
    Retain the notation of Proposition \ref{structureoforbit}, we have 
    $\mathrm{codim}_{\overline{X_J}}(O)=d(\sigma)+d(\tau)+d(\rho)$.
\end{corollary}

Next we describe, for a simple root $\alpha\in \Delta$, how the parabolic subgroups $P_{\alpha}\times_kP^{-}$ and $P\times_kP_{\alpha}^-$ act on closures of $P\times_kP^-$-orbits whose $k$ point is non-empty. As preparation, we record a lemma which concerns the action of simple reflections in $W$ on minimal length representatives.

\begin{lemma} \label{reflectionandminlengthreps}
    Let $(W,S)$ be a Coxeter system and let $I\subset S$. Let $x\in W^I$ and let $s\in S$ be a simple reflection. Then exactly one of the following holds.
    \begin{enumerate}
        \item \label{refmin1} We have $l(sx)<l(x)$. In this case, $sx\in W^I$ always holds.
        \item \label{refmin2} We have $l(sx)>l(x)$ and $sx=xs^{\prime}$ for some simple reflection $s^{\prime}\in I$.
        \item \label{refmin3} We have $l(sx)>l(x)$ and $sx\in W^I$. 
        
    \end{enumerate}
\end{lemma}

\begin{proof}
    Suppose that $l(sx)<l(x)$. If $sx\notin W^I$, then there exists $w\in W_I$ such that $l(sxw)<l(sx)$. Then we have
    \[l(xw)\leq 1+l(sxw)\leq l(sx)<l(x) \]
    which contradicts $x\in W^I$. This finishes case \ref{refmin1}.

    Now suppose that $l(sx)>l(x)$. There will be two cases, depending on whether $sxW_I=xW_I$ or not. If $sxW_I=xW_I$, there exists $w\in W_I$ such that $sx=xw$. Then we have 
    \[l(xw)=l(sx)=l(x)+1\] 
    hence $w=s^{\prime}$ for some $s^{\prime}\in I$. Hence, we are in case \ref{refmin2}.
    
    It remains to consider the case where $l(sx)>l(x)$ and $sxW_I\neq xW_I$. We claim $sx\in W^I$.  If not, there exists $1\neq w\in W^I$ such that $l(sxw)<l(sx)$. First, we show $l(sxw)=l(xw)+1$. Fix a reduced expression 
    \[xw=s_1\ldots s_ts_{t+1}\ldots s_n\]
    such that $x=s_1\ldots s_t$ and $w=s_{t+1}\ldots s_n$. If $l(sxw)<l(xw)$, by exchange property, there exists $1\leq i\leq n$ such that 
    \[sxw=s_1\ldots\widehat{s_i}\ldots s_n.\]
    Because $l(sx)>l(x)$, we must have $t+1\leq i$ which implies $sxW_I=xW_I$, a contradiction. 
    Then we may apply the discussion of case \ref{refmin1} to $sxw$, we see $xw$ is a minimal length representative of $xwW_I=xW_I$ which violates the uniqueness of minimal length representative.
\end{proof}

Now we describe the action of parabolic subgroups.

\begin{proposition}\label{actrk1para}
 For $J\subset \Delta$, let $O=(P\times_k P^-)\cdot (\sigma\rho,\tau)\cdot b_{J}$ be a $P\times_kP^-$-orbit over $k$ in $X_J$, where $\sigma, \tau\in W^J $ and  $\rho\in W_J$. Let $Y=\overline{O}$ be the Zariski closure of $O$. Then for any simple root $\alpha\in \Delta$, we have
    \[(P_{\alpha}\times_kP^-)\cdot Y=\begin{cases}
        \overline{(P\times_k P^-)\cdot (s_{\alpha}\sigma\rho,\tau)\cdot b_{J}} & \text{if \ }l(s_{\alpha}\sigma)<l(\sigma)\\
        \overline{(P\times_k P^-)\cdot (\sigma s_{\beta} \rho,\tau)\cdot b_{J}} & \text{if\ } s_{\alpha}\sigma =\sigma s_{\beta} \text{\ for\ some } \beta\in J \text{\ and\ } l(s_{\beta}\rho)<l(\rho)\\
        Y & \text{otherwise}.
    \end{cases}\]
    Similarly, we have
     \[(P\times_kP_{\alpha}^-)\cdot Y=\begin{cases}
        \overline{(P\times_k P^-)\cdot (\sigma\rho,s_{\alpha}\tau)\cdot b_{J}} & \text{if \ }l(s_{\alpha}\tau)<l(\tau)\\
        \overline{(P\times_k P^-)\cdot (\sigma \rho s_{\beta}, \tau)\cdot b_{J}} & \text{if\ } s_{\alpha}\tau=\tau s_{\beta} \text{\ for\ some } \beta\in J \text{\ and\ } l(\rho s_{\beta})<l(\rho)\\
        Y & \text{otherwise}.
    \end{cases}\]
    Moreover, if $(P_{\alpha}\times_kP^-)\cdot Y\neq Y$, then the action morphism 
    \[(P_{\alpha}\times_kP^-)\times_k^{P\times_kP^-}Y\to (P_{\alpha}\times_kP^-)\cdot Y\] 
    is birational and the same holds for $(P\times_kP_{\alpha}^-)$. 
\end{proposition}

\begin{proof}
   We prove the statement for $P_{\alpha}\times_kP^-$. The proof for the opposite case is similar. 
   By relative Bruhat decomposition, c.f., \Cref{RelativeBruhatforgroup}, we have $P\cdot s_{\alpha}\cdot P=U_{(\alpha)}\cdot s_{\alpha}\cdot P$ is open in $P_{\alpha}$. Let $\dot{s}_{\alpha}, \dot{\sigma}, \dot{\tau}, \dot{\rho}$ be representatives of $s_{\alpha}, \sigma, \tau, \rho$ in $N_G(S)(k)$ respectively. Then, by applying \Cref{structureoforbit} to $O$, we can write
   \begin{equation}\label{1}
       ((Ps_{\alpha}P)\times_kP^{-})\times_k^{P\times_kP^-}O\simeq (U_{(\alpha)}\dot{s}_{\alpha},1) \times_k ((U_{\sigma\rho}\times_k \Big( P^-  \bigcap \tau P_J^-\tau^{-1}\Big))\cdot (\dot{\sigma} \dot{\rho}, \dot{\tau})\cdot b_J).
   \end{equation}
   When $l(s_{\alpha}\sigma\rho)<l(\sigma\rho)$, we have 
   $\Phi^+\bigcap (s_{\alpha}\sigma\rho\Phi^+)=(\alpha)\coprod s_{\alpha}(\Phi^+\bigcap (\sigma\rho)\Phi^+ )$ which implies 
   \[U_{(\alpha)}\dot{s}_{\alpha}U_{\sigma\rho}= U_{s_{\alpha}\sigma\rho}\dot{s}_{\alpha}.\]
   Composing with the action map, the right hand side of (\ref{1}) is mapped to 
   \begin{align*}
       &((U_{{\alpha}}\dot{s}_{\alpha}U_{\sigma\rho}) \times_k \Big( P^-  \bigcap \tau P_J^-\tau^{-1}\Big))\cdot (\dot{\sigma} \dot{\rho}, \dot{\tau})\cdot b_J\\
       =& (U_{s_{\alpha}\sigma\rho}\times_k \Big( P^-  \bigcap \tau P_J^-\tau^{-1}\Big))\cdot(\dot{s}_{\alpha}\dot{\sigma} \dot{\rho}, \dot{\tau})\cdot b_J\\
       =&(P\times_kP^-)\cdot(s_{\alpha}\sigma\rho,\tau)\cdot b_J
   \end{align*}
   which is an isomorphism by \Cref{structureoforbit}. Hence $(P_{\alpha}\times_kP^-)\cdot Y=\overline{(P\times_kP^-)\cdot (s_{\alpha}\sigma\rho,\tau)\cdot b_J}$ and the action map is birational. If $l(s_{\alpha}\sigma)<l(\sigma)$ then $s_{\alpha}\sigma\in W^J$ hence we are in case $1$. If $l(s_{\alpha}\sigma)>l(\sigma)$, because $l(s_{\alpha}\sigma\rho)<l(\sigma\rho)$, we have $s_{\alpha}\sigma=\sigma t$ for some $t\in W_J$. Since $l(s_{\alpha}w)=l(w)+1$, $t=s_{\beta}$ for some $\beta\in J$ and it satisfies $l(s_{\beta}\rho)<l(\rho)$ which gives case $2$.  

   It remains to consider case $3$, which occurs when  $l(s_{\alpha}\sigma\rho)>l(\sigma\rho)$ according to \Cref{reflectionandminlengthreps}. By applying the previous two cases to the orbit closure $\overline{(P\times_kP^-)\cdot(s_{\alpha}\sigma\rho,\tau)\cdot b_J} $, we get
   \[Y=(P_{\alpha}\times_kP^-)\cdot (\overline{(P\times_kP^-)\cdot(s_{\alpha}\sigma\rho,\tau)\cdot b_J})\] 
   hence $(P_{\alpha}\times_kP^-)\cdot Y=Y$.
\end{proof}

\begin{remark}
    The proof of \Cref{actrk1para} indicates that $(P_{\alpha}\times_kP^-)\cdot Y\neq Y$ happens precisely when $l(s_{\alpha}\sigma\rho)<l(\sigma\rho)$. If this is the case, $(P_{\alpha}\times_kP^-)\cdot Y=\overline{(P\times_k P^-)\cdot (s_{\alpha}\sigma\rho,\tau)\cdot b_{J}}$ always holds. We further divide it into two subcases because we want to keep track of how $s_{\alpha}\sigma\rho$ decomposes as a product of elements in $W^J$ and $W_J$.
\end{remark}

For any $G$-variety $Z$, we write 
$$\mathcal{P}(Z)\coloneq\{\text{closed}\; P\text{-stable subvarieties of}\; Z\},$$  
 $$\mathcal{P}_k(Z)\coloneq\{Y \in \mathcal{P}(Z)\vert \ \overline{P\cdot(Y(k))}=Y\},$$ 
 where $\overline{P\cdot(Y(k))}$ means the Zariski closure of $P\cdot(Y(k))$ in $Z$. Note that when $k$ is infinite, $\mathcal{P}_k(Z)$ is simply the set which consists of all $k$-dense closed $P$-stable subvarieties. Also note that $\mathcal{P}(Z)=\mathcal{P}_k(Z)$ when $k$ is an algebraically closed field. By abuse of notation, for any $G\times_k G$-stable subvariety $Z\subset \overline{G}$, we will use $\mathcal{P}(Z),\mathcal{P}_{k}(Z)$ to denote the sets of corresponding $P\times_kP^-$-stable subvarieties. Then by \Cref{borbitproposition} and \Cref{kptsoforbit}, if $Z=\overline{X_J}$ for $J\subset \Delta$, then $\mathcal{P}_k(\overline{X_J})$ consists of elements of the form $\overline{(P\times_kP^{-})\cdot(\sigma\rho,\tau)\cdot b_I}$ for some $I\subset J$, $\sigma,\tau\in W^I$ and $\rho\in W_I$.

\begin{definition}\label{kcancell}
    For a $G$-variety $Z$, we say that the action of $G$ is \textbf{$k$-cancellative} if for $Y_1, Y_2\subset \mathcal{P}_k(Z)$ and $\alpha\in \Delta$ such that $Y_1\neq Y_2$, $P_{\alpha}\cdot Y_1\neq Y_1$ and $P_{\alpha}\cdot Y_2 \neq Y_2$, then $P_\alpha \cdot Y_1 \neq P_{\alpha}\cdot Y_2$.
\end{definition}

When $k$ is algebraically closed (of characteristic zero), then the above definition is the same as \cite[1.2 Definition]{behaviouratinfinityofbruhat}. First we prove a consequence of cancellative action when the base field is algebraically closed.

\begin{lemma}\label{inclusionlemma}
    Suppose that $k$ is algebraically closed (hence $P=B$ is a Borel subgroup). Let $Z$ be a cancellative $G$-variety and let $Y_0\in \mathcal{P}(Z)$. Let $Q\supset B$ be a parabolic subgroup such that $Q\times^{B}Y_0\to Q\cdot Y_0$ is generically finite. Then for any $Y\in \mathcal{P}(Z)$ such that $Q\cdot Y=Q\cdot Y_0$, there exists $w\in W$ such $Y= \overline{B\cdot w\cdot Y_0}$. In particular, $Y_0\subset Y$.
\end{lemma}

\begin{proof}
    Suppose that $Q$ corresponds to $I\subset \Delta$, then we have $Q=\overline{B\cdot w_{0,I}\cdot B}$. First, we show that there exists $w\in W_I$ such that $\overline{B\cdot w\cdot Y}=Q\cdot Y$ and $\overline{B\cdot w\cdot B}\times^BY\to \overline{B\cdot w\cdot Y}$ is generically finite. If $Y=Q\cdot Y$ then we can take $w=1$. When $Y\neq Q\cdot Y$, there exists $\alpha\in I$ such that $P_{\alpha}\cdot Y\neq Y$. Since we assume that $k$ is algebraically closed, in particular $G$ is split, we have $\dim(P_{\alpha})=\dim(B)+1$, which implies $\dim(P_{\alpha}Y)=\dim(Y)+1$ and $P_{\alpha}\times^BY\to P_{\alpha}\cdot Y$ must be generically finite. Note that we have $Q\cdot P_{\alpha} \cdot Y=Q\cdot Y$. By induction on $\mathrm{codim}_{QY}(Y)$, there exists $w_1\in W_I$ such that $Q\cdot Y=\overline{B\cdot w_1\cdot P_{\alpha}\cdot Y}$ and $\overline{B\cdot w_1\cdot B}\times^BP_{\alpha}\cdot Y\to Q\cdot Y$ is generically finite hence taking $w=w_1s_{\alpha_1}$ satisfies the conditions.
    
    Since we have $l(w_{0,I})=l(w)+l(w^{-1}w_{0,I})$, there exists a reduced expression $w_{0,I}=s_{\alpha_1}\ldots s_{\alpha_n}$, where $\alpha_i\in I$ for any $i$, such that $w=s_{\alpha_1}\dots s_{\alpha_t}$ is a reduced expression of $w$ for some $t$. Then the action maps \[P_{\alpha_1}\times^B\ldots\times^BP_{\alpha_t}\times^B\overline{B\cdot w^{-1}w_{0,I}\cdot Y_0}\to Q\cdot Y_0\]
    \[P_{\alpha_1}\times^B\ldots\times^BP_{\alpha_t}\times^BY\to Q\cdot Y\]
    are generically finite. For $1\leq i\leq t$, let $Y_i=P_{\alpha_i}\ldots P_{\alpha_t}\cdot\overline{B\cdot w^{-1}w_{0,I}\cdot Y_0}$ and let $Y_i^{\prime}=P_{\alpha_i}\ldots P_{\alpha_t}\cdot Y$. We prove inductively that $Y_i=Y_i^{\prime}$ for any $i$. If $i=1$ then both $Y_1^{\prime}$ and $Y_1$ equal to $Q\cdot Y$ hence the claim holds. Suppose that we already know $Y_{i}^{\prime}=Y_{i}$. The action maps 
    \[P_{i}\times^BY_{i+1}\to Y_i \text{ and } P_{i}\times^BY_{i+1}^{\prime}\to Y_i^{\prime}\] are generically finite. If $Y_{i+1}\neq Y_{i+1}^{\prime}$, then we will have $Y_i\neq Y_i^{\prime}$ because the action of $G$ is cancellative, hence contradicting our inductive hypothesis.
    Hence $Y_i=Y_i^{\prime}$ for any $i$. Taking $i=t+1$, we conclude that $Y=\overline{B\cdot w^{-1}w_{0,I}\cdot Y_0}$. In particular, we have $Y_0\subset Y$.
\end{proof}

\begin{proposition}\label{existkdensesub}
Let $\alpha\in \Delta$ be a simple root and let $Y\in \mathcal{P}_k(\overline{G})$ such that $(P_{\alpha}\times_kP^-)\cdot Y=Y$. Then
    \begin{enumerate}
        \item There exists a unique $Y_0\in \mathcal{P}_k(\overline{G})$ such that $(P_{\alpha}\times_kP^-)\cdot Y_0=Y$ and $Y_0\neq Y$. In particular, the action of $G\times_kG$ on $\overline{G}$ is $k$-cancellative.
        \item Suppose that $Y^{\prime}\in \mathcal{P}(\overline{G})$ satisfies $(P_{\alpha}\times_kP^-)\cdot Y^{\prime}\supset Y$, then $Y_0\subset Y^{\prime}$.
    \end{enumerate}
Similar statement holds if we replace $P_{\alpha}\times_kP^-$ by $P\times_kP_{\alpha}^-$.

\end{proposition}

\begin{proof}
    We consider the $P_{\alpha}\times_kP^-$ case, the proof for opposite case is similar. Suppose that $(G\times_kG)\cdot Y=\overline{X_I}$ for a subset $I \subset \Delta$ then we may write $Y=\overline{(P\times_k P^-)\cdot (\sigma\rho, \tau)\cdot b_{I}}$ for unique $\sigma, \tau\in W^I$ and $\rho\in W_I$. Because $Y$ is stabilized by $P_{\alpha}\times_kP^-$, by Proposition \ref{actrk1para}, we have 
    \begin{itemize}
        \item either $l(s_{\alpha}\sigma)>l(\sigma)$ and $s_{\alpha}\sigma\in W^I$
        \item or $s_{a}\sigma=\sigma s_{\beta}$ for some $\beta \in I$ and $l(s_{\beta}\rho)>l(\rho)$.
    \end{itemize}
     We take $Y_0=\overline{(P\times_k P^-)\cdot (s_{\alpha}\sigma\rho,\sigma_2)\cdot b_{I}}$ in the first case and take $Y_0=\overline{(P\times_k P^-)\cdot (\sigma s_{\beta}\rho,\tau)\cdot b_{I}}$ in the second case. Then in both case we have $Y_0\neq Y$  and $(P_{\alpha}\times_kP^-)\cdot Y_0=Y$. Moreover, the map $(P_{\alpha}\times_kP^-)\times^{P\times_kP^-}_k Y_0\to Y$ is birational.

     For (2), let $Y^{\prime}\subset \mathcal{P}(\overline{G})$. To show $Y_0\subset Y^{\prime}$, it is enough to prove this after base change to $\bar{k}$. By  (1), the action of $(G\times_k G)_{\bar{k}}$ on $\overline{G
     }_{\bar{k}}$ is cancellative. Let $D\subset Y^{\prime}_{\bar{k}}$ be a closed subvariety such that $(P_{\alpha}\times_kP^-)_{\bar{k}}\cdot D=Y_{\bar{k}}$. Then we may apply Lemma \ref{inclusionlemma} to $(Y_0)_{\bar{k}},\  (P_{\alpha}\times_kP^-)_{\bar{k}},\ D$ and the inclusion follows.
\end{proof}

\begin{remark}
    The statement of the second part of Proposition \ref{existkdensesub} implies that, for any $Y\subset \mathcal{P}(\overline{G})$, the action map $(P_{\alpha}\times_kP^-)\times_k^{P\times_kP^-}Y\to (P_{\alpha}\times_kP^-)\cdot Y$ induces a surjection 
    \[\Big((P_{\alpha}\times_kP^-)\times_k^{P\times_kP^-}Y\Big)(k)\to ((P_{\alpha}\times_kP^-)\cdot Y)(k).\]
    This can fail for general $G$-variety. 
    
    For example, one can take $G=\GL_{2,k}$ and let $Z=G/T$ where $T=\mathrm{Res}_{k^{\prime}/k}(\mathbb{G}_{m,k})$ for some quadratic extension $k^{\prime}/k$. Then the $B_{\bar{k}}$-orbits on $Z_{\bar{k}}$ are in bijection with $T_{\bar{k}}$-orbits on $G_{\bar{k}}/B_{\bar{k}}\simeq \mathbb{P}_{1,\bar{k}}$, hence there are three of them: one open orbit and the other two orbits are closed. Note that the action of $G_{\bar{k}}$ on $Z_{\bar{k}}$ is \emph{not} cancellative. The two closed orbits over $\bar{k}$ are permuted by the Galois action. Hence, over $k$, we have $\mathcal{P}(Z)=\{Z,Y\}$ where $Y$ is the unique closed $B$-stable subvariety defined over $k$ whose base change to $\bar{k}$ is the union of two closed $B_{\bar{k}}$-orbits. Then we have $G\cdot Y=Z\in \mathcal{P}_k(Z)$  and the map $G\times^B Y\to Z$ is generically finite. However, $Y(k)=\emptyset$.
\end{remark}

\subsection{Case 1: comparison in a $G\times_k G$-orbit}
Now we are ready to describe the inclusion relation between different orbit closures. In this part, we treat the case where two $P\times_kP^-$-orbits lie in the same $G\times_kG$-orbit of $\overline{G}$.

\begin{theorem}\label{closureinsamegorbit}
    Let $J\subset \Delta$ be a set of simple roots and let $(\sigma_i\rho_i,\tau_i)$, $i=1,2$ be pairs of elements in $W\times W$ where $\sigma_i,\tau_i\in W^J$ and $\rho_i \in W_J$ for $i=1,2$. Let 
    \[Y_i\coloneq\overline{(P\times_k P^-)\cdot(\sigma_i\rho_i,\tau_i)\cdot b_{J}}\]
    be the corresponding $P\times_k P^-$-orbit closures in $\overline{X_J}$. Then 
    $Y_1\subset Y_2$ if and only if there exists $u\in W_J$ such that 
    \[\sigma_1\rho_1 u\geq \sigma_2\rho_2 \text{\ and\ } \tau_1\geq \tau_2 u^{-1}\]
\end{theorem}

\begin{proof}
    Let $(\sigma_2\rho_2w_0,\tau_2w_{0,J}w_0)=s_1\ldots s_n$ be a reduced expression in $W\times W$. For each $1\leq i\leq n$, denote $Q_i\coloneq P_{\alpha_i}\times_k P^-$ if $s_i=(s_{\alpha_i},1)$ and $Q_i\coloneq P\times_kP_{\alpha_i}^-$ if $s_i=(1,s_{\alpha_i})$, where $\alpha_i\in \Delta$ is a simple root. Then, by \Cref{actrk1para}, we have 
    \[Y_2=Q_1\ldots Q_n\cdot\overline{(P\times_kP^-)\cdot (w_0,w_0w_{0,J})\cdot b_J}.\]
    and the action map \[Q_1\times_k^{P\times_kP^-}\ldots\times_k^{P\times_kP^-}Q_n\times_k^{P\times_kP^-} \overline{(P\times_kP^-)\cdot (w_0,w_0w_{0,J})\cdot b_J}\to Y_2\]
    is birational.
    
    First we claim that, for $Y\in \mathcal{P}_k(\overline{X_J})$ such that $(G\times_kG)\cdot Y=\overline{X_J}$, we have $Y\subset Y_2$ if and only if there exists a subsequence $1\leq i_1<\ldots<i_t\leq n$ (here we allow empty sequence) such that 
    \[Y=Q_{i_1}\ldots Q_{i_t}\cdot \overline{(P\times_kP^-)\cdot (w_0,w_0w_{0,J})\cdot b_J}.\]
    For the "if" part, notice that $Q_i\cdot Z\supset Z$ for any $Z\subset \overline{G}$, hence the result follows from induction on $k$. So it remains to show the “only if” part.
    
    We prove the "only if" part by induction on $n$. When $n=0$ then the claim holds because, by \Cref{dimcomputation},  $\overline{(P\times_kP^-)\cdot(w_0,w_0w_{0,J})\cdot b_J}$ has the smallest dimension among all  $Y\in \mathcal{P}_k(\overline{X_J})$ such that $(G\times_k G)\cdot Y=\overline{X_J}$. When $n>0$, let 
    \[Y_2^{\prime}=Q_2\ldots Q_n\cdot \overline{(P\times_kP^-)\cdot (w_0,w_0w_{0,J})\cdot b_J}.\] Then we have $Y_2^{\prime}\neq Y_2 $  and $Q_1\cdot Y_2^{\prime}=Y_2$. Suppose that $Y\in \mathcal{P}_k(\overline{X_J})$ is an element such that $(G\times_kG)\cdot Y=\overline{X_J}$ and $Y\subset Y_2$. If $Q_1\cdot Y\neq Y$, then by \Cref{existkdensesub}, $Y\subset Y_2^{\prime}$. Therefore, the claim follows from the inductive hypothesis in this case. If $Q_1\cdot Y= Y$, let $Y_0$ be the unique element in $\mathcal{P}_k(\overline{G})$ such that $Q_1\cdot Y_0=Y$ and $Y_0\neq Y$. Then $Y_0\subset Y_2^{\prime}$ and  $(G\times_kG)\cdot Y_0=(G\times_k G)\cdot Y=\overline{X_J}$. Hence, we can apply the inductive hypothesis to $Y_0$ to deduce the claim.
    
     Suppose that $Y_1\subset Y_2$ and let $s_{i_1},\ldots,s_{i_t}$ be a sequence of simple reflections from the previous claim.  Choose a sequence such that $k$ is minimal, then the action map
    \[Q_{i_1}\times_k^{P\times_kP^-}\ldots \times_k^{P\times_kP^-}Q_{i_t}\times_k^{P\times_{k}P^-} \overline{(P\times_kP^-)\cdot (w_0,w_0w_{0,J})\cdot b_J}\to Y_2\]
    is birational by \Cref{actrk1para}. Let $(x,y)\coloneq s_{i_1}\ldots s_{i_t}\in W\times W$, which is a reduced expression by the minimality of $t$. Therefore, we have $x\leq \sigma_2\rho_2w_0,\  y\leq \tau_2w_{0,J}w_0$ and $Y_{1}=\overline{(P\times_k P^-)\cdot (xw_0,yw_{0}w_{0,J})\cdot b_J}$ by \Cref{actrk1para}. By the uniqueness part of \Cref{borbitproposition}, there exists $u\in W_J$ such that $xw_0=\sigma_1\rho u$ and $yw_{0}w_{0,I}=\tau_1u$. Then we have $\sigma_1\rho_1u=xw_0\geq \sigma_2\rho_2$ and $\tau_1uw_{0,J}=yw_0\geq \tau_2w_{0,J}$.
    
    To show the "if" part of \Cref{closureinsamegorbit} holds, it remains to show that the latter inequality is equivalent to $\tau_1\geq \tau_2u^{-1}$. Notice that for any $a,b\in W$ such that $a\geq b$, if $s\in W$ is a simple reflection that satisfies $l(sb)<l(b)$, then we have $sa\geq sb$. Indeed, if $l(sa)>l(a)$ then we have $sa\geq a\geq sb$. If $l(sa)<l(a)$, choose a reduced expression $a=s_1\ldots s_n$ with $s_1=s$. Since $b\leq a$, there is a reduced subword such that $b=s_{i_1}\ldots s_{i_t}$. If $i_1=1$ then we have $sb=s_{i_2}\ldots s_{i_t}\leq s_2\ldots s_n=sa$. If $i_1\neq 1$, then we have  $b\leq sa$ hence $sb\leq sa$ as well. 
    Now we fix a reduced expression $uw_{0,J}=s_1\ldots s_t$. Then apply the previous claim $t$-times implies $\tau_1\geq \tau_2u^{-1}$.

    The "only if" part of \Cref{closureinsamegorbit} follows by reversing the above process.
\end{proof}

\subsection{Case 2: comparison in two $G\times_k G$-orbits}
In this part, we treat the case where two $P\times_kP^-$-orbits lie in two different $G\times_kG$-orbits of $\overline{G}$ respectively.

For any $P\times_kP^-$-stable closed subscheme $D\subset \overline{G}$, we denote $$\mathcal{C}(D) \coloneq \bigcup_{Y\in\mathcal{P}_k(\overline{ G}), Y\subset D }Y.$$ 
Note that $\mathcal{C}(D)$ is just $\overline{D(k)}$ when $k$ is an infinite field. We introduce this notation in order to treat both finite fields and infinite fields cases uniformly.

Let $I\subset \Delta$ and let $Z=\overline{X_I}$ be the corresponding orbit closure. By definition, any $Y_0\in \mathcal{P}_k(Z)$ such that $Y_0\subset Y$ satisfies $Y_0\subset \mathcal{C}(Y\bigcap Z)$. To understand the inclusion relation between orbit closures, in view of \Cref{closureinsamegorbit}, it is enough to describe the irreducible components of $\mathcal{C}(Y\bigcap Z)$. When $k$ is algebraically closed, such a description is first obtained in \cite[2.1 Theorem]{behaviouratinfinityofbruhat}. Our next goal is to generalize Brion's result to arbitrary base fields.

\begin{theorem}\label{irrecompofintersection}
    Suppose that $Y=\overline{(P\times_k P^-)\cdot (\sigma\rho,\tau)\cdot b_{J}}$  where $\sigma,\tau \in W^J $ and  $\rho\in W_J$. Let $I\subset \Delta$ be a subset and let $Z=\overline{X_I}$ be the corresponding $G\times_k G$-orbit closure. We have  
    \begin{enumerate}
        \item If $Y\bigcap X_I\neq \emptyset$, then $I\subset J$.
        \item Suppose that $I \subset J \subset \Delta$. Then we have a decomposition into irreducible components: 
    \[\mathcal{C}\Big(Y\bigcap Z\Big)=\bigcup \overline{(P\times_kP^-)\cdot(\sigma\rho v,\tau v)\cdot b_I} \]
    where the union is over all $v\in W_J\bigcap W^I$ satisfying $l(\sigma\rho)=l(\sigma\rho v)+l(v)$. 
    
    Moreover, each $\overline{(P\times_kP^-)\cdot (\sigma\rho v,\tau v)\cdot b_I}$ that appears in the above decomposition is an irreducible component of $Y\bigcap Z$.
    \end{enumerate}

\end{theorem}

First, we prove the (1) part of \Cref{irrecompofintersection}.

\begin{lemma}
    If $I\subset \Delta$ is a subset such that $I \not \subset J$. Then $Y\bigcap X_I=\emptyset$.
\end{lemma}

\begin{proof}
Suppose that $Y\bigcap X_I\neq \emptyset$. Since $Y\subset \overline{X_J}$, this implies $\overline{X_J}\bigcap X_I\neq \emptyset$. Since $\overline{X_J}\bigcap X_I$ is $G\times_k G$-stable, we must have $X_I\subset \overline{ X_J}$, a contradiction.
\end{proof}
 
Next, we prove a lemma about how irreducible components of $\mathcal{C}(Y\bigcap Z)$ behave under the action of parabolic subgroups, which will be used in the proof of (2) of \Cref{irrecompofintersection}.

\begin{lemma}\label{remainirrecomp}
    Retain the notation of Theorem \ref{irrecompofintersection}. Suppose that $C\subset  \mathcal{C}(Y\bigcap Z)$ is an irreducible component and $\alpha\in \Delta$ is a simple root such that $(P_{\alpha}\times_k P^-)\cdot C\neq C$. Then $(P_{\alpha}\times_kP^{-})\cdot C$ is an irreducible component of $\mathcal{C}((P_{\alpha}\times_kP^{-}\cdot Y)\bigcap Z)$. The similar statement holds if we replace $P_{\alpha}\times_kP^-$ by $P\times_kP_{\alpha}^-$.
\end{lemma}

\begin{proof}
    Because $C$ is an irreducible component of $\mathcal{C}(Y\bigcap Z)$, $C\in \mathcal{P}_k(\overline{G})$ and hence $(P_{\alpha}\times_k P^-)\cdot C\in \mathcal{P}_k(\overline{G})$ as well. Because we assume $(P_{\alpha}\times_k P^-)\cdot C\neq C$, we have  $(P_{\alpha}\times_k P^-)\cdot C\not\subset Y\bigcap Z$ as $C$ is an irreducible component of $\mathcal{C}(Y\bigcap Z)$. In particular, we have $(P_{\alpha}\times_k P^-)\cdot Y\neq Y$ and it has a smaller codimension in $\overline{X_J}$.

    Let  $C^{\prime}$ be an irreducible component of   $\mathcal{C}((P_{\alpha}\times_k P^-)\cdot Y)\bigcap Z)$ such that $C^{\prime}\supset(P_{\alpha}\times_k P^-)\cdot C$. Since both $Z$ and $(P_{\alpha}\times_k P^-)\cdot Y$ are $P_{\alpha}\times_k P^-$-stable, $C^{\prime}$ must also be $P_{\alpha}\times_k P^-$-stable. By Proposition \ref{existkdensesub}, there exists a unique $C_0\in\mathcal{P}_k(\overline{G})$ such that $(P_{\alpha}\times_k P^-)\cdot C_0\neq C_0$ and $(P_{\alpha}\times_k P^-)\cdot C_0=C^{\prime}$. Since $(P_{\alpha}\times_k P^-)\cdot C\subset C^{\prime}=(P_{\alpha}\times_k P^-)\cdot C_0$, we also have $C\subset C_0$. Because 
    \[(P_{\alpha}\times_k P^-)\times_k^{P\times_kP^-}\Big(Y\bigcap Z\Big)\to (P_{\alpha}\times_k P^-)\cdot \Big(Y\bigcap Z\Big)= ( (P_{\alpha}\times_k P^-)\cdot Y)\bigcap Z\]
    is surjective, we have $C_0\subset Y\bigcap Z$ by the second part of Proposition \ref{existkdensesub}. Therefore, we must have $C_0=C$ and $(P_{\alpha}\times_k P^-)\cdot C=C^{\prime}$ is an irreducible component of $\mathcal{C}((P_{\alpha}\times_kP^{-})\cdot  Y)\bigcap Z)$.
\end{proof}

During the proof of (2) Theorem \ref{irrecompofintersection}, we will also need the following result concerning the properness of intersection from \cite[1.4 Theorem(ii)]{behaviouratinfinityofbruhat}. In \emph{loc.cit.}, Brion assumes that the base field is $\mathbb{C}$. For self-completeness, we reproduce his proof here in the wonderful compactification case.

\begin{lemma}\label{properintersection}
    For $Y,Z$ as in Theorem \ref{irrecompofintersection}, $Y$ intersects $Z$ properly in $\overline{X_J}$. 
\end{lemma}

\begin{proof}
    Since $Y_{\bar{k}}$ is a $(B\times_{\bar{k}} B^{-})$-orbit closure, it is enough to prove the lemma when $k=\bar{k}$. Suppose that $Z=Z_1\bigcap\ldots\bigcap Z_c$ where $Z_i$ are codimension one $(G\times_kG)_{\bar{k}}$-orbits in $(G\times_kG)_{\bar{k}}\cdot Y$ and $c=\mathrm{codim}_{\overline{X_J}}(Z)$. Let $i$ be the largest index such that the intersection $Y\bigcap Z_1 \bigcap\ldots\bigcap Z_i$ is proper. If $i\neq c$, then there exists an irreducible component $C\subset Y\bigcap Z_1\ldots \bigcap Z_i$ such that $C\subset Z_{i+1}$. Let $(w_1,w_2)\in W\times W$ such that $\overline{(B\times_{\bar{k}}B^-)\cdot (w_1,w_2)\cdot C}=(G\times_kG)_{\bar{k}}\cdot C$ and the action map
    \[\overline{(B\times_{\bar{k}}B^-)\cdot (w_1,w_2)\cdot (B\times_{\bar{k}}B^{-})}\times_{\bar{k}}^{B\times B^-}C\to (G\times_kG)_{\bar{k}}\cdot C\]
    is birational hence $\dim(C)+l(w_1)+l(w_2)=\dim((G\times_kG)_{\bar{k}}\cdot C)$. Because we assume that the intersection $Y\bigcap Z_1\bigcap \ldots \bigcap Z_i$ is proper, we have $\dim(C)=\dim(Y)-i$. Since $C\subset Z_{i+1}$, we have $(G\times_kG)_{\bar{k}}\cdot C\subset Z_1\bigcap \ldots \bigcap Z_{i+1}$ thus $\dim((G\times_kG)_{\bar{k}}\cdot C)\leq \dim(X_J)-i-1$. Hence, we have 
    \[\dim(\overline{(B\times_{\bar{k}}B^-)\cdot (w_1,w_2)\cdot Y})\leq l(w_1)+l(w_2)+\dim(Y)=l(w_1)+l(w_2)+\dim(C)-i\leq \dim(X_J)-1.\]
    Therefore, $\overline{B\times_{\bar{k}}B^-\cdot (w_1,w_2)\cdot Y}$ is not $(G\times G)_{\bar{k}}$-stable. Then there exists an irreducible $B\times_{\bar{k}}B^-$-divisor $Y^{\prime}\subset \overline{X_J}$ such that $Y^{\prime}\supset \overline{(B\times_{\bar{k}}B^-)\cdot (w_1,w_2)\cdot Y}$ and that $Y^{\prime}$ is not $(G\times_kG)_{\bar{k}}$-stable (i.e., $Y^{\prime}$ is a color). But any such divisor lies in the complement of $(\mathscr{R}_u(B) \times_{\bar{k}} \mathscr{R}_u(B^-)) \cdot \Big(\widetilde{T_{\overline{k}}} \bigcap \overline{X_J}\Big) \subset \overline{G}_{\bar{k}}$. On the other hand, any $(G\times_k G)_{\bar{k}}$-orbit contains a base point $b\in \widetilde{T_{\overline{k}}}$ by \Cref{propositionofbasepoint}. Therefore, such $Y^{\prime}$ cannot contain $(G\times_kG)_{\bar{k}}$-orbit, hence a contradiction. Therefore $i=c$.
\end{proof}

\begin{proof}[Proof of \Cref{irrecompofintersection}]
    Retain the notation of \Cref{irrecompofintersection}. If $Y=\overline{X_J}$, then we have \[\overline{X_J}=\overline{(P\times_k P^-)\cdot (1,1)\cdot b_J}\text{ and }Y\bigcap Z=Z=\overline{(P\times_k P^-)\cdot (1,1)\cdot b_I}.\] Hence the theorem holds in this case. 

    Now assume $\mathrm{codim}_{\overline{X_J}}(Y)>0$. By induction, we may assume that the statement of \Cref{irrecompofintersection} holds for $Y^{\prime}, Z^{\prime}$ when either of the following condition holds:
    \begin{itemize}
        \item $(G\times_kG)\cdot Y^{\prime}\subsetneq (G\times_kG)\cdot Y ;$
        \item $(G\times_kG)\cdot Y^{\prime}=(G\times_kG)\cdot Y=\overline{X_J}$ and $\mathrm{codim}_{\overline{X_J}}(Y^{\prime})<\mathrm{codim}_{\overline{X_J}}(Y);$
        \item $Y=Y^{\prime}$ and $Z^{\prime}\subsetneq Z$.
    \end{itemize} 
    
    Let  $C=\overline{(P\times_kP^-)\cdot(\sigma\rho v,\tau v)\cdot b_I}$ such that $v\in W_J\bigcap W^I$ satisfying $l(\sigma\rho)=l(\sigma\rho v)+l(v)$. First we show that $C\subset \mathcal{C}(Y\bigcap Z)$ and is an irreducible component of $Y\bigcap Z$. By \Cref{dimcomputation}, we have 
    \[\mathrm{codim}_{\overline{X_I}}(C)=d(\sigma\rho v)+d(\tau v)=d(\sigma)+d(\rho)+d(\tau)=\mathrm{codim}_{\overline{X_{J}}}(Y)>0.\] 
    Therefore $C$ is not $G\times_kG$-stable and there exists $\alpha\in \Delta$ such that $(P_{\alpha}\times_kP^-)\cdot C\neq C$ or $(P\times_kP_{\alpha}^-)\cdot C\neq C$. We deal with the first possibility, the proof in the opposite case is analogous. By \Cref{actrk1para}, we have    
    \begin{enumerate}
        \item either $l(s_{\alpha}\sigma)<l(\sigma)$
        \item or $s_{\alpha}\sigma=\sigma s_{\beta}$ for some $\beta\in I $ such that $ l(s_{\beta}\rho v)<l(\rho v)$.
    \end{enumerate}
    If (1) occurs, we have $(P_{\alpha}\times_kP^-)\cdot Y\neq Y$. If $(2)$ happens, we have
    \[l(s_{\beta}\rho)=l(s_{\beta}\rho v)+l(v)<l(\rho v)+l(v)=l(\rho)\]
    hence $(P_{\alpha}\times_kP^-)\cdot Y\neq Y$ as well. Then by inductive hypothesis, $(P_{\alpha}\times_kP^-)\cdot C$ is an irreducible component of $\mathcal{C}(((P_{\alpha}\times_kP^-)\cdot Y)\bigcap Z)$. Since 
    \[(P_{\alpha}\times_kP^-)\cdot (Y\bigcap Z)=((P_{\alpha}\times_kP^-)\cdot  Y)\bigcap Z \]
    by Proposition \ref{existkdensesub} we conclude $C\subset \overline{Y\bigcap Z}$. Moreover, $C\subset \mathcal{C}(Y\cap Z)$ because $C\in \mathcal{P}_k(\overline{G})$. By \Cref{properintersection}, $Y$ intersects $Z$ properly inside $(G\times_kG)\cdot Y=\overline{X_J}$. Since $\mathrm{codim}_Z(C)=\mathrm{codim}_{\overline{X_J}}(Y)$, we see that $C$ must be an irreducible component of $Y\bigcap Z$.
    
    Now let $C^{\prime}\subset \mathcal{C}(Y\bigcap Z)$ be an irreducible component. Next we show that $C^{\prime}$ is not $G\times_kG$-stable and $(G\times_kG)\cdot C^{\prime}=Z$. If $(G\times_kG)\cdot C^{\prime} \subsetneq Z$, by inductive hypothesis, we have 
    \[C^{\prime}=\overline{(P\times_kP^-)\cdot (\sigma \rho v^{\prime},\tau v^{\prime})\cdot b_{I^{\prime}}}\] for some $I^{\prime}\subsetneq I$ and $v^{\prime}\in W_J\bigcap W^{I^{\prime}}$ such that $l(\rho)=l(\rho v^{\prime})+l(v^{\prime})$. Let $v_1\in W^I$ and $v_2\in W_I$ be the unique elements satisfying $v^{\prime}=v_1v_2$. Then $l(\rho)=l(\rho v_1)+l(v_1)$ and $l(\rho v_1)=l(\rho v)+l(v_2)$.
    By the previous paragraph, $Y^{\prime}=\overline{(P\times_kP^-)\cdot (\sigma \rho v_1,\tau v_1)\cdot b_{I}}\subset Y\bigcap Z$ and $C^{\prime}\subset Y^{\prime}\bigcap ((G\times_kG)\cdot C^{\prime})$ by inductive hypothesis. But this is impossible because we choose $C^{\prime}$ to be an irreducible component of $\mathcal{C}(Y\bigcap Z)$. Hence we conclude $(G\times_kG)\cdot C^{\prime}=Z$. Because $Y$ intersects $Z$ properly inside $\overline{X_J}$ by \Cref{properintersection}, we have $\mathrm{codim}_Z(C^{\prime})\geq\mathrm{codim}_{\overline{X_J}}(Y)>0$ hence $C^{\prime}$ cannot be $G\times_k G$-stable. 
    
    Because $C^{\prime}$ is not $G\times_kG$-stable, there exists a simple root $\alpha\in \Delta$ such that either $(P_{\alpha}\times_kP^-)\cdot C^{\prime} \neq C^{\prime}$ or $(P\times_kP_{\alpha}^-)\cdot C^{\prime}\neq C^{\prime}$. We treat the first possibility while the second follows from a similar argument.
    By Lemma \ref{remainirrecomp}, $(P_{\alpha}\times_kP^-)\cdot C^{\prime}$ is an irreducible component of $((P_{\alpha}\times_kP^-) \cdot Y)\bigcap Z$. Since $(P_{\alpha}\times_kP^-) \cdot Y\neq Y$ we have that

     \begin{enumerate}
        \item either $l(s_{\alpha}\sigma)<l(\sigma)$
        \item or $s_{\alpha}\sigma=\sigma s_{\beta}$ for some $\beta\in I $ and $ l(s_{\beta}\rho )<l(\rho )$.
    \end{enumerate}
    By inductive hypothesis, there exists $v$ such that 
    \[(P_{\alpha}\times_kP^-)\cdot C^{\prime}=\begin{cases}
        \overline{(P\times_kP^-)\cdot(s_{\alpha}\sigma\rho v,\tau v)\cdot b_I} & \text{if (1) holds}\\
        \overline{(P\times_kP^-)\cdot (\sigma s_{\beta}\rho v,\tau v)\cdot b_I} & \text{if (2) holds}
    \end{cases} 
    \] 
    and $l(\rho)=l(\rho v)+l(v)$ (resp., $l(s_{\alpha}\rho )=l(s_{\alpha}\rho v)+l(v)$ ) if (1) (resp., (2)) holds.  In either case, we can deduce that $l(v)=l(\rho v)+l(v)$ and $C^{\prime}=\overline{(P\times_kP^-)\cdot (\sigma\rho v,\tau v)\cdot b_I}$ hence the theorem follows.
\end{proof}

\subsection{Conclusion}
Now we can combine \Cref{closureinsamegorbit} and \Cref{irrecompofintersection} to deduce the main result of \Cref{sectionbruhatorder}.

\begin{theorem}\label{relbruhatmainthm}
    Given $Y_1,Y_2\in \mathcal{P}_k(\overline{G})$ such that 
    $$Y_1=\overline{(P\times_kP^-)\cdot (\sigma_1\rho_1,\tau_1)\cdot b_{I_1}} \;\text{and}\;Y_2=\overline{(P\times_kP^-)\cdot (\sigma_2\rho_2,\tau_2)\cdot b_{I_2}},$$
where $\sigma_i,\tau_i\in W^{I_i}$ and $\rho_i\in W_{I_i}$, for $i=1,2$. Then $Y_1\subset Y_2$ if and only if $I_1\subset I_2$ and there exist two elements $v\in W_{I_2}\bigcap W^{I_1}$, $u\in W_{I_1}$ such that $\sigma_1\rho_1u\geq \sigma_2\rho_2v \text{, } \tau_1\geq \tau_2vu^{-1}$
and $l(\rho)=l(\rho v)+l(v)$.
\end{theorem}

\begin{remark}
    Theorem \ref{relbruhatmainthm} is a generalization of \cite[2.4 Proposition]{intersectioncohomologyofBorbitclosures} to arbitrary base field. In \emph{loc.cit.}, Springer works with $B\times_{\bar{k}}B$-orbit closures. Using the relation $B^-=\tilde{w}_{0}B\tilde{w}_0$ where $\tilde{w}_0\in\widetilde{W}$ is the longest element, it is not hard to see our theorem is equivalent to Springer's result when $k$ is algebraically closed.
\end{remark}

\section{Topological closure}
We keep the notations of \Cref{sectionbruhatorder}.
Suppose that $k$ is endowed with a non-discrete topology $\mathcal{E} $, satisfying
the separation axiom $T_1$ and with respect to which it is a topological ring.
Then for any $k$-variety $Z$, the set $Z(k)$ is equipped with a natural topology, see, for instance, \cite{Conradtopology} for a discussion of basic properties of this topology. For example, this assignment is functorial in the sense that any morphism between $k$-varieties $f:Z_1\to Z_2$ induces a continuous map $Z_1(k)\to Z_2(k)$. Moreover, because $\mathcal{E}$ is $T_1$, points in $k$ are closed subsets. Therefore, for any closed subvariety $Z^{\prime}\subset Z$, $Z^{\prime}(k)$ is closed in $Z(k)$. Hence the topology on $Z(k)$ is finer than the one induced from Zariski topology.

By \cite[Corollary 21.28]{Borellinearalgebraicgroup}, if we equip $G(k)$ with the topology induced by $\mathcal{E}$, then the closure of $(PwP^{-})(k)$ in $G(k)$ coincides with $\overline{PwP^-}(k)$.

We equip $\overline{G}(k)$ with the topology $\mathcal{T}$ induced by $\mathcal{E}$. Our goal in this section is to generalize the above statement to wonderful compactification.

\begin{theorem}\label{theoremtopologicalclosure}
  Let $O=(P\times_k P^-)\cdot (\sigma\rho,\tau)\cdot b_{J}$, where $\sigma,\tau\in W^J $ and  $\rho\in W_J$. Then the closure of $O(k)$ in $\overline{G}(k)$ with respect to the topology $\mathcal{T}$ equals $\overline{O}(k)$, where the bar indicates taking Zariski closure.
\end{theorem}

First, we show that the closure relation for $(P(k)\times P^-(k))$-orbits inside each $(G(k)\times G(k))$-orbit coincides with the relation given by Zariski closure, which is analogous to \Cref{closureinsamegorbit} in the topological setting.

\begin{lemma}\label{6.2}
    Let $O=(P\times_k P^-)\cdot(\sigma\rho,\tau)\cdot b_{J}$ be a $(P\times_kP^{-})$-orbit defined over $k$, where $\sigma,\tau\in W^J $ and  $\rho\in W_J$. Let  $Q=P_{\alpha}\times_kP^-$ or $P\times_kP_{\alpha}^-$ for some $\alpha\in \Delta$. Suppose that $Q\cdot \overline{O}\neq \overline{O}$. Let $O^{\prime}$ be the unique open $P\times_kP^{-}$-orbit defined over k inside $Q\cdot O$. Then the closure of $O^{\prime}(k)$ inside $\overline{G}(k)$ with respect to $\mathcal{T}$ contains $O(k)$ 
\end{lemma}

\begin{proof}
    We work with the $Q=P_{\alpha}\times_kP^-$ case while the other case is similar. Let $\dot{\sigma},\dot{\rho},\dot{\tau}$ be representatives of $\sigma, \rho, \tau$ in $G(k)$ respectively. Consider the subvariety $K=(U_{-(\alpha)}\times_k\{1\})(\dot{\sigma}\dot{\rho},\dot{\tau})b_J$. Because $Q\overline{O}\neq \overline{O}$, the action map induces an isomorphism $U_{\alpha}\simeq K$. Since $U_{(\alpha)}$ is isomorphic to an affine space over $k$, c.f., \cite[21.20~Theorem~(i)]{Borellinearalgebraicgroup}, therefore $K$ is also isomorphic to an affine space over $k$. For $u\in U_{-(\alpha)}(k)$, if $u\neq 1$ then $u\in (Ps_{\alpha}P)(k)$ and hence $(u,1)(\dot{\sigma}\dot{\rho},\dot{\tau})b_J\in O^{\prime}(k)$ when $u\neq1$. Since the topology on $k$ is non-discrete, we see that $(\dot{\sigma}\dot{\rho},\dot{\tau})b_J$ lies in the closure of $O^{\prime}(k)$ hence $O(k)$ lies in the closure of $O^{\prime}(k)$ as well.
\end{proof}

\begin{corollary}\label{6.3}
    Let $O_i\coloneq(P\times_k P^-)\cdot (\sigma_i\rho_i,\tau_i)\cdot b_{J}$, $i=1,2$ be two $P\times_kP^-$-orbits. Suppose that $O_1\subset \overline{O_2}$, then $O_1(k)$ is contained in the closure of $O_2(k)$ with respect to $\mathcal{T}$.
\end{corollary}

\begin{proof}
    Suppose that $(G\times_kG)\cdot\overline{O}=\overline{X_J}$ for $J\subset \Delta$. From the proof of Theorem \ref{closureinsamegorbit}, we see that there exists a sequence of parabolic subgroups $Q_1,...,Q_n$, where each $Q_i=P_{\alpha_i}\times_kP^-$ or $P\times_kP^-_{\alpha_i}$ for some $\alpha_i\in \Delta$, such that \[\overline{O_2}=Q_1\ldots Q_n\cdot \overline{(P\times_kP^-)\cdot (w_0,w_0w_{0,J})\cdot b_J}\]
    and there is a subsequence $i_1,\ldots,i_t$ such that
    \[\overline{O_1}=Y_2=Q_{i_1}\ldots Q_{i_t}\cdot \overline{(P\times_kP^-)\cdot (w_0,w_0w_{0,J})\cdot b_J}.\]
    Then the corollary follows from \Cref{6.2}. 
\end{proof}

Next, we show the analogue of Theorem \ref{irrecompofintersection} in topological setting. Because $\mathcal{T}$ is finer than the Zariski topology, it is enough to show that, for each $C\subset \mathcal{C}(Y\bigcap Z)$ is an irreducible component, $C(k)$ is contained in the closure of $O(k)$.

\begin{lemma}\label{6.4}
      Let $O=(P\times_k P^-)\cdot (\sigma\rho,\tau)\cdot b_{J}$  where $\sigma,\tau\in W^J $ and  $\rho\in W_J$. Let $I \subset J \subset \Delta$ and let  $v\in W_J\bigcap W^I$ be an element satisfying $l(\sigma\rho)=l(\sigma\rho v)+l(v)$. Let $O^{\prime}=(P\times_k P^-)\cdot (\sigma\rho v,\tau v)\cdot b_{I}$. Then $O^{\prime}(k)$ is contained in the closure of $O(k)$ with respect to $\mathcal{T}$.
\end{lemma}

\begin{proof}
    Since the closure of $O(k)$ with respect to $\mathcal{T}$ is stable under $P(k)\times P^{-}(k)$-action, it is enough to show that $(\dot{\sigma}\dot{\rho} \dot{v},\dot{\tau} \dot{v})b_I$ lies in the closure of $O(k)$, where $\dot{\sigma},\dot{\rho}, \dot{v},\dot{\tau} $ are representatives of $\sigma, \rho, v, \tau$ in $N_G(S)(k)$ respectively. 
    
    Because $v\in W_J$, we have $\dot{v}\in N_{L_J}(S)(k)\subset L_J(k)$. Recall from \Cref{propositionofbasepoint}, there exists a cocharacter $\lambda\in X_{*}(S)$ such that there exists a map $f:\mathbb{A}_{1,k}\to \overline{X_J}$ which satisfies: $f|_{\mathbb{G}_{m,k}}$ is given by $t\to (\lambda(t),1)b_J$ and $f(0)=b_I$. Note that for any $t\in k^{\times}$, we have $f(t)=(\dot{\sigma}\dot{\rho} \dot{v},\dot{\tau} \dot{v})(\lambda(t),1)b_J=(\dot{\sigma}\dot{\rho} ,\dot{\tau})((v\lambda)(t),1)b_{J}\in O$. Hence $(\dot{\sigma}\dot{\rho} \dot{v},\dot{\tau} \dot{v})b_I=(\dot{\sigma}\dot{\rho} \dot{v},\dot{\tau} \dot{v})f(0)$ lies in the closure of $O(k)$. 
\end{proof}

Combining \Cref{6.3} ,\Cref{6.4} and the description of $\overline{O}$ (\Cref{relbruhatmainthm}) proves \Cref{theoremtopologicalclosure}.

\renewcommand{\bibname}{References}

\printbibliography 

\end{document}